\documentclass[11pt, a4paper]{article}

\usepackage{amsmath,amssymb,amsthm,mathrsfs}
\usepackage{float}
\usepackage[utf8]{inputenc}
\usepackage[english]{babel}
\usepackage[T1]{fontenc}
\usepackage{lmodern}
\usepackage[totalwidth=160mm, totalheight=250mm]{geometry}

\usepackage[auth-sc, affil-it]{authblk}
\usepackage{tikz}
\usepackage{pstricks}
\usepackage{tikz-3dplot}
\usepackage{pgf}
\usepackage{curve2e}

\renewcommand{\geq}{\geqslant}

\newcommand{\bbZ}{\mathbb{Z}}
\newcommand{\bbQ}{\mathbb{Q}}
\newcommand{\bbR}{\mathbb{R}}

\newcommand{\bbC}{\mathbb{C}}
\newcommand{\bbD}{\mathbb{D}}

\newcommand{\Dc}{\mathcal{D}}

\newcommand{\Mc}{\mathcal{M}}

\newcommand{\Rc}{\mathcal{R}}

\newcommand{\cwt}{{\bf Cut Wire Theorem}}

\theoremstyle{plain}
\newtheorem{main-theorem}{Theorem}
\newtheorem{theo}{Theorem}[section]
\newtheorem{prop}[theo]{Proposition}
\newtheorem{lemm}[theo]{Lemma}

\theoremstyle{definition}
\newtheorem{defi}[main-theorem]{Definition}
\newtheorem{remark}[main-theorem]{Remark}
\newtheorem{exam}[theo]{Example}
\newtheorem{rema}[theo]{Remark}

\newtheorem*{clai-nn}{Claim}
\newtheorem*{theorem*}{Theorem}
\newtheorem*{schoencond*}{Sch\"onflies Condition}

\author[1]{Beno\^it Loridant\thanks{This author was supported by the Agence Nationale de la Recherche (ANR), the Japanese Society for the Promotion of Science (JSPS) and the Austrian Science Fund
(FWF) through the projects ANR-FWF I1136 and JSPS-FWF I3346.}}

\author[2]{Jun Luo\thanks{Supported by the Chinese National Natural Science Foundation Projects 11771391 and 11871483.}}
\author[3]{Yi Yang}

\affil[1]{  Montanuniversit\"at Leoben,
    Franz Josefstrasse 18, Leoben 8700, Austria
}
\affil[2]{School of Mathematics,
    Sun Yat-Sen University, Guangzhou 510275, China
}
\affil[3]{School of Mathematical Sciences, Peking University, Beijing 100871, China}

\affil[3]{Corresponding author: yangyi@math.pku.edu.cn or yangyi\_1989@163.com}

\title{\Large A Core Decomposition of Compact Sets in the Plane}
\date{\today}

\begin{document}
\maketitle

\begin{abstract}

A Peano continuum means a locally connected continuum. A compact metric space is called a  \emph{Peano compactum} if all its components are Peano continua and if for any constant $C>0$ all but finitely many of its components are of diameter less than $C$.  Given a compact set $K\subset\mathbb{C}$, there usually exist several upper semi-continuous decompositions of $K$ into subcontinua such that the quotient space, equipped with the quotient topology, is a Peano compactum. We prove that one of these decompositions is finer than all the others and call it the
\emph{core decomposition of $K$ with Peano quotient}. This core decomposition gives rise to a metrizable quotient space, called the Peano model of $K$, which is shown to be determined by the topology of $K$ and hence independent of the embedding of $K$ into $\bbC$. We also construct a concrete continuum $K\subset\mathbb{R}^3$ such that the core decomposition of $K$ with Peano quotient does not exist. For specific choices of $K\subset\bbC$, the above mentioned core decomposition coincides with two models obtained recently, namely the locally connected model for unshielded planar continua (like connected Julia sets of polynomials) and the finitely Suslinian model for unshielded planar compact sets (like polynomial Julia sets that may not be connected). The study of such a core decomposition provides partial answers to several questions posed by Curry in 2010. These questions are motivated by other works, including those by Curry and his coauthors, that aim at understanding the dynamics of a rational map $f: \hat{\bbC}\rightarrow\hat{\bbC}$ restricted to its Julia set.

\textbf{Keywords.} \emph{Peano compactum, locally connected, finitely Suslinian, core decomposition.}
\end{abstract}



\newpage


\section{Introduction and main results}

In this paper, a compact metric space is called a \emph{compactum} and a connected compactum is called a \emph{continuum}. If $K,L$ are two compacta, a continuous onto map $\pi:K\to L$ such that the preimage of every point in $L$ is connected is called \emph{monotone} \cite{Whyburn79}.
We are interested in compacta in the complex plane $\bbC$ or in the Riemann sphere $\hat{\bbC}$.

Given a compactum $K\subset\mathbb{C}$, an \emph{upper semi-continuous decomposition} $\mathcal{D}$ of $K$ is a partition of $K$ such that for every open set $B\subset K$ the union of all $d\in\mathcal{D}$  with $d\subset B$ is open in $K$ (see~\cite{Kuratowski68}). Let $\pi$ be the natural projection sending $x\in K$ to the unique element of $\mathcal{D}$ that contains $x$. Then a set $A\subset \mathcal{D}$ is said to be open in $\mathcal{D}$ if and only if  $\pi^{-1}(A)$ is open in $K$. This defines the \emph{quotient topology} on $\mathcal{D}$. If all the elements of such a decomposition $\Dc$ are compact the equivalence on $K$ corresponding to $\Dc$ is  a closed subset of $K\times K$ and the quotient topology on $\Dc$ is metrizable \cite[p.148, Theorem 20]{Kelley55}. Equipped with an appropriate metric compatible with the quotient topology, the quotient space $\Dc$ is again a compactum.

An upper semi-continuous decomposition of a compactum $K\subset\mathbb{C}$ is \emph{monotone} if each of its elements is a subcontinuum of $K$.  In this case, the mapping $\pi$ described above is a monotone map. Let $\mathcal{D}$ and $\mathcal{D}'$ be two monotone
decompositions of a compactum  $K\subset\mathbb{C}$, with projections $\pi$ and $\pi'$, and suppose that $\mathcal{D}$ and $\mathcal{D}'$ both satisfy a  topological property $(T)$.  We say that $\mathcal{D}$ is \emph{finer than} $\mathcal{D}'$ \emph{with respect to} $(T)$ if  there is a map $g:\mathcal{D}\to\mathcal{D}'$ such that $\pi'=g\circ \pi$. If a monotone
decomposition of a compactum $K$ is finer than every other monotone decomposition of $K$ with respect to $(T)$, then it is called the \emph{core decomposition} of $K$ with respect to $(T)$ and it will be denoted by $\mathcal{D}_K^T$ or simply $\mathcal{D}_K$, if $(T)$ is fixed. Clearly, the core decomposition $\mathcal{D}_K^T$ is unique, if it exists.

Recently, core decompositions with respect to two specific topological properties were studied for the special class of compacta $K\subset\mathbb{C}$ that are {\em unshielded}, that is,  $K=\partial W$ for the unbounded component $W$ of  $\mathbb{C}\setminus K$. More generally, when dealing with subsets of the Riemann sphere, a compactum $K\subset\hat{\mathbb{C}}$ is called unshielded if $K=\partial W$ for some component $W$ of $\hat{\mathbb{C}}\setminus K$. In particular, if a rational function $R:\hat{\mathbb{C}}\rightarrow\hat{\mathbb{C}}$ has a completely invariant Fatou component, then its Julia set is an unshielded compactum (see~\cite[Theorem 5.2.1.(i)]{Beardon91}). It is known that every polynomial Julia set is unshielded, while a rational Julia set may not be. 

In 2011, Blokh-Curry-Oversteegen obtained the core decomposition $\mathcal{D}_K^{LC}$ of an arbitrary unshielded continuum $K$ with respect to the property of local connectedness \cite[Theorem 1]{BCO11}. A special case is when $K$ is the connected Julia set of a polynomial. In our paper, we will solve the existence of $\mathcal{D}_K^{LC}$ for all continua $K\subset\bbC$, without assuming that $K$ is unshielded. In particular,  our  core decomposition will apply to the study of connected Julia sets of rational functions on the extended complex plane $\hat{\bbC}$. Such a result is very helpful in searching for an answer to \cite[Question 5.2]{Curry10}. However, if we only consider upper semi-continuous decompositions then there might be two decompositions $\mathcal{D}_1,\mathcal{D}_2$ of an unshielded continuum $K\subset\mathbb{C}$ which are both Peano continua under quotient topology, such that the only decomposition finer than $\mathcal{D}_1$ and $\mathcal{D}_2$ is the decomposition $\{ \{z\}: z\in K\}$ into singletons. Actually, let $\mathcal{C}\subset[0,1]\subset\mathbb{C}$ be Cantor's ternary set.
Let $K$ be the union of $\{x+{\mathbf i}y: x\in \mathcal{C}, y\in[0,1]\}$ with $\{x+{\mathbf i}: x\in [0,1]\}$ and call it the {\em Cantor Comb}.  See the following figure for an approximation of $K$.
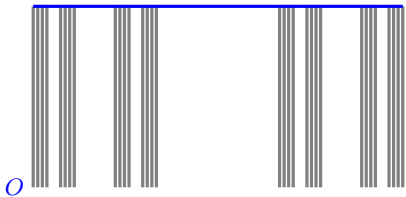
\begin{figure}[ht]
\begin{center}
\begin{tikzpicture}[scale=0.2]
\footnotesize
 \pgfmathsetmacro{\xone}{0}
 \pgfmathsetmacro{\xtwo}{24.3}
 \pgfmathsetmacro{\yone}{0}
 \pgfmathsetmacro{\ytwo}{12}

\foreach \x in {0,5.4,16.2,21.6}
  \foreach \p in {0,1.8}
   \foreach \q in {0,.3,.6,.9}
  \draw[gray, very thick] (\x+\p+\q,0) -- (\x+\p+\q,\ytwo);

\draw[blue,very thick] (0,\ytwo) -- (\xtwo,\ytwo);
 \draw[blue] (0,0) node[anchor=east] {$O$};
\end{tikzpicture}
\end{center}\vskip -0.75cm
\caption{An approximation of the Cantor Comb.}
\end{figure}
Let $p_1$ be the restriction  to $K$ of the projection $x+{\mathbf i}y\mapsto x$. Then $\mathcal{D}_1=\left\{p_1^{-1}(x): x\in[0,1]\right\}$ coincides with $\mathcal{D}_K^{LC}$, the core decomposition of $K$ with respect to local connectedness. Let $\mathcal{D}_2$ be the union of all the translates $C_y:=\{x+{\mathbf i}y: x\in\mathcal{C}\}$ of $K$ with $0\le y\le 1$ and all the single point sets $\{z=x+{\mathbf i}\}$ with $x\notin\mathcal{C}$. Then $\mathcal{D}_2$ is an upper semi-continuous decomposition of $K$, which is not monotone and which is a Hawaiian earing under quotient topology. Clearly, the only decomposition finer than both $\mathcal{D}_1$ and $\mathcal{D}_2$ is the decomposition $\{ \{x\}: x\in K\}$ into singletons.

In 2013, Blokh-Curry-Oversteegen obtained for any unshielded compactum $K\subset\mathbb{C}$ the core decomposition $\mathcal{D}_{K}^{FS}$ with respect to the property of being finitely Suslinian \cite[Theorem 4]{BCO13}. Here a compactum is \emph{finitely Suslinian} provided that every collection of pairwise disjoint subcontinua whose diameters are bounded away from zero is finite. Since every finitely Suslinian continuum is locally connected,
we see that \cite[Theorem 1]{BCO11} is a special case of \cite[Theorem 4]{BCO13}.
We may wonder about the existence of the core decomposition $\mathcal{D}_{K}^{FS}$ of an arbitrary compactum  $K\subset\mathbb{C}$. However, there are examples of continua $K\subset\mathbb{C}$ failing to have such a core decomposition (see~\cite[Example 14]{BCO13} and Section~\ref{sec:lclamcore} of this paper). We will replace the property of being finitely Suslinian by the property of being a {\it Peano compactum}. This class of compacta will be defined below. For every compactum $K\subset \mathbb{C}$, we will prove the existence of the core decomposition $\mathcal{D}_K^{PC}$ with respect to the property of being a Peano compactum. We will briefly call $\mathcal{D}_K^{PC}$ the {\it core decomposition of $K$ with Peano quotient}. Since a finitely Suslinian compactum is also a Peano compactum, the decomposition $\mathcal{D}_K^{PC}$ is finer than $\mathcal{D}_K^{FS}$ for any compactum  $K\subset\mathbb{C}$, when the latter exists. For unshielded compacta  $K\subset\mathbb{C}$, we will prove in Section~\ref{sec:lclamcore} that our core decomposition $\mathcal{D}_K^{PC}$  coincides with the core decomposition $\mathcal{D}_K^{FS}$, given by~\cite[Theorem 4]{BCO13}.

Clearly, the core decompositions $\mathcal{D}_K^{PC}$ and $\mathcal{D}_K^{LC}$ for a continuum $K\subset\mathbb{C}$ must coincide with each other.  However,  for a general continuum or compactum $K$ that can not be embedded into the plane, the core decomposition $\Dc_K^{PC}$ may not exist. We refer to Example \ref{no-CD-ex} for an example
of continuum in $\bbR^3$ which has no core decomposition $D_K^{PC}$.

Our work is motivated by recent studies in and possible applications to the field of complex dynamics, but we will rather focus on the topological part.  The  concept of Peano compactum has its origin in an ancient result of Sch\"onflies. It also has motivations from some recent works by Blokh, Oversteegen and their colleagues. We will show that this property can be used advantageously in discussing core decompositions, besides the properties of being locally connected or finitely Suslinian. Sch\"onflies' result reads as follows.

\begin{theorem*}\cite[p.515, $\S61$, II, Theorem 10]{Kuratowski68}. If  $K$ is a locally connected compactum in the plane and if the sequence $R_1,R_2,\ldots$ of components of $\bbC\setminus K$ is infinite, then the sequence of their diameters converges to zero.
\end{theorem*}

The above theorem gives a necessary condition for planar compacta to be locally connected. This condition is also necessary for planar compacta to be  finitely Suslinian, as we will prove in Theorem~\ref{FS-necessary}. However, in both cases, the condition is not sufficient. For instance, Sierpinski's universal curve is not finitely Suslinian but its complement has infinitely many components whose diameters converge to zero. Also, the closed topologist's sine curve is not locally connected but its complement has a single component. This motivates us to introduce the following condition, which happens to be necessarily fulfilled  by every planar compactum $K$, if $K$ is assumed to be finitely Suslinian or locally connected.

\begin{schoencond*} A compactum $K$ in the plane fulfills the \emph{Sch\"onflies condition} if for the unbounded strip $U=U(L_1,L_2)$ bounded by any two parallel lines $L_1$ and $L_2$, the {\bf difference} $\overline{U}\setminus K$ has at most finitely many components intersecting both $L_1$ and $L_2$.
\end{schoencond*}

\begin{main-theorem}\label{necessary}
If a compactum $K$ in the plane is locally connected or finitely Suslinian then it satisfies the Sch\"onflies condition.
\end{main-theorem}

We will obtain an equivalent formulation of the Sch\"onflies condition in Lemma~\ref{separation}, that is: {\em a compactum $K\subset\bbC$ satisfies the Sch\"onflies condition if and only if for any strip $U=U(L_1,L_2)$ bounded by  two parallel lines $L_1$ and $L_2$, the {\bf intersection} $\overline{U}\cap K$ has at most finitely many components that intersect both $L_1$ and $L_2$}. Moreover, we will show that the above Sch\"onflies condition entirely characterizes the local connectedness for continua in the plane. A very similar characterization is obtained in \cite[Theorem 1.1]{Luo07}, which uses open annuli bounded by two concentric circles on a planar continuum.

\begin{main-theorem}\label{LC-criterion}
A continuum $K$ in the plane is locally connected if and only if it satisfies the Sch\"onflies condition.
\end{main-theorem}

Continua like Sierpinski's universal curve indicate that the above theorem does not hold if we replace {\em locally connected} with {\em finitely Suslinian}.

The main purpose of this paper is not to characterize the finitely Suslinian compacta, but to find appropriate candidates for the core decomposition, instead of the finitely Suslinian property. Such a core decomposition of planar compacta will have interesting applications to the study of Julia sets of rational functions. See for instance the open questions proposed by Curry at the end of \cite{Curry10}. Combined with earlier models developed by Blokh-Curry-Oversteegen \cite{BCO11,BCO13},  the results of Theorem~\ref{necessary} and Theorem~\ref{LC-criterion} provide some evidence that planar compacta satisfying the Sch\"onflies condition seem to be a reasonable model for the above mentioned core decompositions. Therefore, we give a nontrivial characterization of the Sch\"onflies condition as follows.

\begin{main-theorem}\label{sufficient}
A compactum $K\subset\bbC$ satisfies the Sch\"onflies condition if and only if it has the following two properties.
\begin{itemize}
\item[(1)] Every component of $K$ is a Peano continuum.
\item[(2)] For every  $C>0$, all but finitely many   components of $K$ are of diameter less than $C$.
\end{itemize}
\end{main-theorem}

A compact metric space satisfying the above two properties in Theorem~\ref{sufficient} will be called a \emph{Peano compactum}.
In particular, a Peano continuum is a connected Peano compactum. After the above preparations, we are in a good position to introduce our strategy to prove the existence of $\Dc_K^{PC}$, the core decomposition with Peano quotient, for every compactum $K$ in the plane.

Let us define a relation $R_K$ on any given compactum $K\subset\bbC$ as follows, together with the smallest closed equivalence containing $R_K$. To this end, for two disjoint simple closed curves $J_1$ and $J_2$ we denote by $U(J_1,J_2)$ the component of $\hat{\bbC}\setminus(J_1\cup J_2)$ whose boundary equals $J_1\cup J_2$. This is an annulus  in the extended complex plane $\hat{\bbC}$.

\begin{defi}\label{model}
Let $K\subset\mathbb{C}$ be a compactum. We say that two points $x,y\in K$ are {\em related under $R_K$} provided that either $x=y$ or there exist two disjoint simple closed curves $J_1\ni x$ and $J_2\ni y$ such that $\overline{U(J_1,J_2)}\cap K$ contains an infinite sequence of components $Q_n$ intersecting both $J_1$ and $J_2$, whose limit $\lim\limits_{n\rightarrow\infty}Q_n$ under Hausdorff distance contains $\{x,y\}$. Moreover, let $\mathscr{R}_K$ be the collection of all the closed equivalence relations on $K$ containing $R_K$ (as subsets of $K\times K$). The intersection of all the equivalence relations in $\mathscr{R}_K$, denoted as $\sim$, is the smallest closed equivalence containing $R_K$ hence is  also an element of $\mathscr{R}_K$. We call $R_K$ the \emph{Sch\"onflies relation on $K$} and $\sim$ the \emph{Sch\"onflies equivalence on $K$}.
\end{defi}

One can check that the equivalence class $[x]$ under the Sch\"onflies equivalence $\sim$ is a continuum for every $x\in K$ (see Proposition \ref{connected-class}). Moreover, if $K$ in Theorem~\ref{sufficient} is assumed to be unshielded, then  it is finitely Suslinian if and only if it satisfies the Sch\"onflies condition (see Theorem~\ref{FS}). Also, if the compactum $K$ in Definition~\ref{model} is unshielded, then $\sim$ coincides with the relation developed in~\cite{BCO13}, given by the finest finitely Suslinian model (see Section~\ref{sec:lclamcore}).

From now on, we denote by $\mathcal{D}_K$ the collection of equivalence classes $[x]:=\{z\in K: z\sim x\}$, for $x$ running through $K$. Thus $\mathcal{D}_K$ is the decomposition induced by the Sch\"onflies equivalence. It is standard to verify that the hyperspace $\mathcal{D}_K$ is necessarily a compact, Hausdorff and secondly countable space under quotient topology \cite[p.148, Theorem 20]{Kelley55}. Therefore, it is metrizable by Urysohn's metrization theorem \cite[p.125, Theorem 16]{Kelley55} and we may consider it as a compactum. We will prove that it satisfies the two properties of Theorem~\ref{sufficient}, thus is a  Peano compactum.

\begin{main-theorem}\label{quotient}
Under quotient topology $\mathcal{D}_K$  is a  Peano compactum.
\end{main-theorem}

After this we will show that $\mathcal{D}_K$ is finest in the following sense.

\begin{main-theorem}\label{finest}
Let $\sim$ be the Sch\"onflies equivalence  on a compactum $K\subset\bbC$ and $\pi(x)=[x]$ the  natural projection from $K$ to $\mathcal{D}_K$. If $f: K\rightarrow Y$ is monotone map onto a  Peano compactum $Y$, then there is an onto map $g: \mathcal{D}_K\to Y$ with $f=g\circ \pi$.
\end{main-theorem}

By the above Theorems~\ref{quotient} and~\ref{finest}, we can conclude that  the core decomposition $\mathcal{D}_K^{PC}$ with Peano quotient exists and coincides with $\mathcal{D}_K$.

\begin{main-theorem}\label{mainaim}
Every compactum $K\subset\mathbb{C}$ has a core decomposition $\mathcal{D}_K^{PC}$ with respect to the property of being a Peano compactum. It coincides with the decomposition $\mathcal{D}_K$ induced by the Sch\"onflies equivalence $\sim$ on $K$.
\end{main-theorem}

Finally, we want to emphasize that the core decomposition with Peano quotient may not exist for a general compactum $K$ (see Example \ref{no-CD-ex}). However, if it exists then it is unique and is determined by the topology of $K$ (see Theorem \ref{topological-CD}). Combining this with the result of Theorem \ref{mainaim}, we see that for any planar compactum $K$ the core decomposition $\Dc_K^{PC}$ always exists and is independent of the embedding of $K$ into the plane. We will also refer to the quotient space $\Dc_K^{PC}$  as the {\em Peano model of $K$}.

\begin{remark}\label{rmk:RelationCurry} Theorem~\ref{mainaim} is very helpful when one tries to  answer \cite[Questions 5.2 and 5.4]{Curry10}. Indeed, in the first part of \cite[Question 5.4]{Curry10}, Curry asks: for what useful topological properties (P) does there exist a finest decomposition of every Julia set $J(R)$ (of a rational function $R$) satisfying (P)? Our Theorem \ref{mainaim} indicates that the property of ``being a Peano compactum'' is such a property.  Moreover, in the last part of \cite[Question 5.4]{Curry10}, Curry asks:  which of these (properties) is the appropriate analogue for the finest locally connected model? Since the core decomposition $\mathcal{D}_K^{PC}$ in Theorem \ref{mainaim} generalizes the earlier finest decompositions obtained in~\cite{BCO11,BCO13}, the property of ``being a Peano compactum'' is a reasonable candidate. What remains is to verify that the decomposition $\Dc_K^{PC}$ is ``dynamic'', in the sense that the rational function $R$ sends every element of $\Dc_K^{PC}$ into an element of $\Dc_K^{PC}$. This is what Curry asks, in \cite[Question 5.2]{Curry10} and in the middle part of \cite[Question 5.4]{Curry10}. We expect further work on the decomposition $\Dc_K$ in this direction,  especially towards applications to the dynamical study of a rational function restricted to its Julia set (see for instance \cite{Yang17}).

\end{remark}

\begin{remark}\label{rmk:degen}
For any compactum $K$, the decomposition $\{Q: Q\ \text{is\ a\ component\ of}\ K\}$ always induces a Peano quotient, whose components are single points. Therefore, an important problem is to determine whether the core decomposition $\mathcal{D}_K^{PC}$  of  a given compactum $K\subset\mathbb{C}$ induces a quotient space having a non-degenerate component. In such a case, we say that $\mathcal{D}_K^{PC}$ is a {\it non-degenerate} core decomposition; otherwise, we say that it is a {\it degenerate} core decomposition. Clearly, the core decomposition of every indecomposable continuum $K\subset\mathbb{C}$ is degenerate, since $\mathcal{D}_K^{PC}=\{K\}$. The studies of Blokh-Curry-Oversteegen~\cite{BCO13,BCO11}, whose models are generalized by the core decomposition introduced in this paper, already provide very interesting results on  the existence of non-degenerate core decompositions. For instance, by~\cite[Theorem 27]{BCO11}, if a continuum $X\subset K$ has a ``well-slicing family'', then the image of $X$ under the natural projection $\pi: K\rightarrow\mathcal{D}_K^{PC}$ is a non-degenerate continuum, hence $K$ has a non-degenerate core decomposition.  If $K$ is the Julia set of a polynomial, then it is stated in~\cite[Corollary 24]{BCO13} that $K$ has a  non-degenerate core decomposition $\mathcal{D}_K^{PC}$ if and only if $K$ has a periodic component $Q$ which, as a plane continuum, has a non-degenerate core decomposition $\mathcal{D}_Q^{PC}$. In other words, to compute the core decomposition $\mathcal{D}_K^{PC}$ we just need to compute the core decomposition $\mathcal{D}_Q^{PC}$ for all the periodic components $Q$ of $K$. If the above Julia set $K$ is connected and is ``finitely irreducible'', the result of \cite[Theorem 4.1]{Curry10} indicates that the core decomposition $\mathcal{D}_K^{PC}$ satisfies either $\mathcal{D}_K^{PC}=\{K\}$ or $\mathcal{D}_K^{PC}=\{\{x\}: x\in K\}$. Actually, in the former case $K$ is an indecomposable continuum and in the latter case it is  homeomorphic to $[0,1]$. Finally, if  $X$ is an unshielded continuum and $Y\subset X$ is a subcontinuum, Blokh-Oversteegen-Timorin \cite{BOT17} obtained recently a sufficient condition for the core decomposition $\Dc_Y^{PC}$ of $Y$ to be embedded canonically into that of $X$. As an application to complex dynamics, the authors also considered the special case that $X$ is the connected Julia set of a renormalizable polynomial $P$ and $Y$ is the so-called {\em small Julia set}, for a polynomial-like map obtained as a restriction of some iterate $P^n$ with $n>1$. Combining these results with the core decomposition obtained in our paper, one may investigate problems like the local connectedness of the Julia set of an infinitely renormalizable polynomial.
\end{remark}

We arrange our paper as follows. Section~\ref{sec:lclamcore} briefly recalls facts on local connectedness, laminations in complex dynamics and core decompositions. In this section we also prove that each unshielded Peano compactum is finitely Suslinian (see Theorem~\ref{FS}). From this, one can infer that the core decompositions $\mathcal{D}_K^{PC}$ and $\mathcal{D}_K^{FS}$ of an unshielded compactum $K\subset\mathbb{C}$ are equal. Section~\ref{lemmas} gives preliminary lemmas needed in the proofs of the main theorems. Section~\ref{sec:necsuff} proves Theorems \ref{necessary} to \ref{sufficient}. Sections~\ref{sec:ps} and~\ref{sec:cd}  respectively prove Theorems~\ref{quotient} and~\ref{finest}. Finally, in Section 7  we construct a concrete continuum $K\subset\bbR^3$ having no core decomposition with Peano quotient. In this section, we also show that $\Dc_K^{PC}$ is entirely determined by the topology of $K$, if it exists (see  Theorem \ref{topological-CD}).

\textbf{Acknowledgments.} The authors are grateful to the referee for the suggestions: they greatly improved the paper. In particular, we thank the referee for pointing out (1) an interesting  question leading to Remark~\ref{rmk:RelationCurry}; (2) the context of dynamical systems for Curry's questions proposed in \cite{Curry10}; (3) the necessity to provide an example of a compactum which does not have a core decomposition with Peano quotient; and (4) the exact references of well-known results in topology from which some of our preliminary lemmas or results in Section~\ref{lemmas} could be deduced. The rearrangement of our lemmas in Section~\ref{lemmas} according to the referee's suggestions actually allowed us to simplify our arguments, especially in the proof of Theorem~\ref{quotient-con} and  Theorem~\ref{finest}.


\section{Local Connectedness, Lamination, and Core Decomposition}\label{sec:lclamcore}

The investigation of local connectedness dates back to the nineteenth century. Cantor proved that the unit interval and the unit square have the same cardinality. In other words, there exists a bijection $h:[0,1]\rightarrow[0,1]^2$, and this map $h$ can not be continuous. Peano and some of his contemporaries further obtained continuous surjections from $[0,1]$ onto planar domains like squares and triangles. The range of a continuous map from $[0,1]$ into a metric space is therefore often called a \emph{Peano continuum}. Peano continua were then fully characterized via the notion of local connectedness: indeed, Hahn and Mazurkiewicz showed that a continuum is a Peano continuum if and only if it is locally connected.

Among the Peano continua of the plane, the boundary of a bounded simply connected domain $U$ provides a special case. By the Riemann Mapping Theorem, there is a conformal isomorphism from the unit open disk $\bbD=\{|z|<1\}$ onto $U$. Furthermore, Carath\'eodory's theorem states that this  conformal mapping has a continuous extension to the closed disk $\overline{\bbD}$ if and only if the boundary $\partial U$ is locally connected. Considering $U$ as a domain in the extended complex plane $\hat{\bbC}$, we may assume, after the action of a M\"obius map, that $\infty\in U$. Then $X=\mathbb{C}\setminus U$ is a \emph{full}  continuum, \emph{i.e.}, it has a connected complement $U$. Moreover,  there is a conformal isomorphism $\Phi$ from $\bbD^*=\{|z|>1\}\cup\{\infty\}$ onto $U$, fixing  $\infty\in\hat{\bbC}$ and having a real derivative at $\infty$.

In the study of quadratic dynamics, examples of the above map $\Phi$ are (1) B\"ottcher maps for hyperbolic polynomials $z\mapsto z^2+c$ with $c$ lying in a hyperbolic component of the Mandelbrot set $\Mc$ and (2) the conformal isomorphism sending $\bbD^*$ onto  $\hat{\bbC}\setminus\Mc$.

For the  map $\Phi$ in (1), the boundary of $\Phi(\bbD^*)$ is the Julia set $J_c$ of $z\mapsto z^2+c$, which is known to be locally connected. In this case,  $J_c$ is the image of the unit circle $\partial\bbD=\partial\bbD^*$ under a continuous map (called Carath\'eodory's loop), hence may be considered as the quotient space of an equivalence on $\partial\bbD$. This equivalence is a {\em lamination} in Thurston's sense \cite{Thurston09}. Douady \cite{Douady93} proposed a pinched disc model describing full locally connected continua in the plane. Extending the lamination in a natural way to a closed equivalence relation $\mathcal{L}$ on the closed unit disk $\overline{\bbD}$,  he obtains that $K_c$ is homeomorphic with the quotient $\overline{\bbD}/\!\mathcal{L}$, where $K_c$ is the filled Julia set of the polynomial $z\mapsto z^2+c$ ($J_c=\partial K_c$).

The pinched disc model still works, even if the full continuum $K$ is not locally connected. The map $\Phi$ in (2) provides a typical example, in  which the boundary of $\Phi(\bbD^*)$ coincides with that of the Mandelbrot set $\Mc$. Denote by $\Rc_\theta$ the image of $\{re^{2\pi\theta{\bf i}}: r>1\}$ under $\Phi$ for $\theta\in[0,1]$. $\Rc_\theta$ is called the {\em external ray at $\theta$}. If $\lim\limits_{r\rightarrow 1}\Phi\left(re^{2\pi\theta{\bf i}}\right)$ is a point on $\partial\Mc$, denoted as $c_\theta$, we say that $\Rc_\theta$ \emph{lands at} $c_\theta$. It is known that all external rays $\Rc_\theta$ with rational $\theta$ must land. Douady therefore \cite{Douady93} defines an equivalence relation on
$\displaystyle \left\{e^{2\pi\theta{\bf i}}: \theta\in\bbQ\cap[0,1]\right\}$ by setting $\theta\sim^\bbQ_\Mc\theta'$ if and only if $c_\theta=c_{\theta'}$. As a subset of $\partial\bbD\times\partial\bbD$, the closure of $\sim^\bbQ_\Mc$ (denoted $\sim_\Mc$) turns out to be an equivalence relation on $\partial\bbD$ (see \cite[Theorem 3]{Douady93} for fundamental properties of $\sim_\Mc$).

Let us now recall the main ideas of Blokh-Curry-Oversteegen \cite{BCO11} concerning locally connected models for unshielded continua in the plane. Let $K\subset\hat{\bbC}$ be an unshielded continuum  with $K=\partial U$, where $U$ is the unbounded component of $\bbC\setminus K$. Let $\Phi$ be a conformal mapping that sends $\bbD^*$ to $U$ and fixes $\infty$. For any $\theta\in[0,1]$, the {\em impression} at $e^{2\pi\theta{\bf i}}$, defined by
\[{\rm Imp}(\theta)=\left\{\lim\limits_{i\rightarrow\infty}\Phi(z_i): \ \{z_i\}\subset\bbD^*,\  \lim\limits_{i\rightarrow\infty}z_i=e^{2\pi\theta{\bf i}}\right\},\]
is a subcontinuum of $K$. By \cite[Lemma 13]{BCO11} there is a minimal closed equivalence $\mathcal{I}$ on $K$ whose classes are made up of prime end impressions. By \cite[Lemma 16]{BCO11}, if $\Rc$ is an arbitrary closed equivalence on $K$ such that the quotient space $K/\Rc$ is a locally connected continuum then $\mathcal{I}$ is contained in $\Rc$ (as subsets of $K\times K$). The first part of \cite[Lemma 17]{BCO11} obtains that the quotient $K/\!\mathcal{I}$ is a locally connected continuum, called the locally connected model of $K$. Now we may define a closed equivalence relation $\sim_{K}$ on $\partial\bbD$ by requiring that $\theta\sim_{K}\theta'$ if and only if ${\rm Imp}(\theta)$ and ${\rm Imp}(\theta')$ lie in the same equivalence class $[x]_\mathcal{I}$. Then, by the second part of \cite[Lemma 17]{BCO11}, the equivalence $\sim_{K}$ is a lamination such that the induced  quotient $\partial\bbD/\!\sim_K$ is homeomorphic to $K/\!\mathcal{I}$.

In particular, when $K$ is the Julia set of a polynomial $f$ of degree $d\ge2$ without irrationally neutral cycles, Kiwi \cite{Kiwi04} investigates the structure of the classes $[x]_\mathcal{I}$ and shows that every $[x]_\mathcal{I}$ coincides with the {\em fiber at $x\in K$} \cite[Definition 2.5]{Kiwi04} defined by
\[{\rm Fiber}(x)=\left\{y\in K: \ \text{no\ finite\ set\ of\ eventually\ periodic\ points\ separates}\ y\ \text{and}\  x\ \text{in}\  K\right\}.\]
Here, a finite set $C\subset K$ \emph{separates} two points of $K$ if these points are in distinct components of $K\setminus C$. We refer to~\cite[Corollary 3.14]{Kiwi04} and \cite[Proposition 3.15]{Kiwi04} for important properties of ${\rm Fiber}(x)$, and to Schleicher's works \cite{Schleicher99a,Schleicher99b,Schleicher04} for another approach to the study of fibers.

To generalize the above model, Blokh, Curry and Oversteegen \cite{BCO13} define an equivalence $\simeq$ on an unshielded compactum $K\subset\mathbb{C}$ to be the minimal closed equivalence such that every limit continuum is contained in a single class $[x]_\simeq:=\{z\in K: z\simeq x\}$. Recall that a limit continuum is the limit $\lim\limits_{k\rightarrow\infty}N_k$ under Hausdorff distance of an infinite sequence of pairwise disjoint subcontinua $N_k\subset K$. The quotient space $\mathcal{D}_K^{FS}=\{[x]_\simeq: x\in K\}$ is necessarily a compact metrizable space \cite[p.38, Theorem 3.9]{Nadler92}. The authors of \cite{BCO13} further show that it is finitely Suslinian
\cite[Lemma 13]{BCO13}.

Every element $d$ of the above decomposition $\mathcal{D}_K^{FS}$,  as a subset of $\mathbb{C}$,  possesses the following  property: the union of all the bounded components of $\mathbb{C}\setminus d$ does not intersect $K$. The authors of \cite{BCO13}  then employ Moore's theorem to prove that $\mathcal{D}_K^{FS}$ is the finest monotone decomposition of $K$ with finitely Suslinian quotient \cite[Theorem 19]{BCO13}. In other words, $\mathcal{D}_K^{FS}$ is the core decomposition of $K$ with respect to the finitely Suslinian property.

Let now $K\subset\mathbb{C}$ be an arbitrary compactum. Let $\mathcal{D}_K^{\simeq}=\{[x]_\simeq: x\in K\}$ with $\simeq$ defined as above. On the other side, let $\sim$ be the Sch\"onflies equivalence on $K$, defined in Definition~\ref{model} as the minimal closed equivalence  relation containing the relation $R_K$. We write $\mathcal{D}_K=\{[x]_\sim: x\in K\}$. We want to compare these two decompositions.  The definition of $R_K$ indicates that if $(z_1,z_2)\in R_K$ then there is a limit continuum containing both $z_1$ and $z_2$. This in turn indicates that $\sim$ is contained in $\simeq$ as subsets of $K\times K$, hence the decomposition $\mathcal{D}_K$ always refines $\mathcal{D}_K^{\simeq}$.

These two decompositions turn out to be equal provided that $K$ is unshielded. Note that in this case $\mathcal{D}_K^{\simeq}=\mathcal{D}_K^{FS}$ is the core decomposition of $K$ with respect to the finitely Suslinian property \cite[Theorem 19]{BCO13}. Actually, the unshielded assumption of $K$ implies that the bounded components of $\mathbb{C}\setminus d$ for every $d\in\mathcal{D}_K$ are all disjoint from $K$. Let $d^*$ be the union of $d$ with the bounded components of $\mathbb{C}\setminus d$. Then
\[\mathcal{D}_{\mathbb C}:=\left\{d^*: d\in\mathcal{D}_K\right\}\cup\left\{\{z\}: z\notin\left(\bigcup_{d\in\mathcal{D}_K} d^*\right)\right\}\]
is a monotone decomposition of $\mathbb{C}$, such that $d_1^*\cap d_2^*=\emptyset$ for any $d_1\ne d_2\in\mathcal{D}_K$. By Moore's Theorem, the quotient $\mathcal{D}_{\mathbb C}$ is homeomorphic to the plane and the natural projection $\Pi: \mathbb{C}\rightarrow\mathcal{D}_{\mathbb C}$ sends $K$ to a planar compactum. Since every $d^*$ is disjoint from the unbounded component $W$ of $\mathbb{C}\setminus K$, the image $\Pi(W)$ is a region in the plane $\mathcal{D}_{\mathbb C}$ whose boundary contains $\Pi(K)$. That is to say, $\Pi(K)$ is also an unshielded compactum in $\mathcal{D}_{\mathbb C}$, which is homeomorphic with the plane $\bbC$. On the other hand, for any $x,y\in K$ it is direct to check that $\Pi(x)=\Pi(y)$ if and only if $\pi(x)=\pi(y)$. Therefore, the quotient $\mathcal{D}_K$ is homeomorphic to $\Pi(K)$ hence may be embedded into the plane as an unshielded compactum. Theorem \ref{quotient} of this paper says that $\mathcal{D}_K$ is also a Peano compactum. By the following theorem, such a planar compactum is finitely Suslinian. Consequently, the core decomposition decomposition $\mathcal{D}_K^{FS}$ is finer than $\mathcal{D}_K$, and we have  $\mathcal{D}_K^{FS}=\mathcal{D}_K$.

\begin{theo}\label{FS} Each unshielded Peano compactum $K\subset\mathbb{C}$ is finitely Suslinian. \end{theo}
\begin{proof} By Theorem \ref{sufficient} and the definition of Peano compactum, we only need to consider the simpler case when $K$ is actually a continuum. Recall that a continuum $X$ is {\em regular} at a point $x\in X$ if for every neighborhood $V_x$ of $x$ there exists a neighborhood $U_x\subset V_x$ of $x$ whose boundary $\partial U_x=\overline{U_x}\cap\overline{X\setminus U_x}$ is a finite set \cite[p.19]{Whyburn42}. A {\em regular} continuum is just one that is regular at each of its points. Here it is standard to check that a regular continuum is finitely Suslinian. Therefore, our proof will be completed if only we can verify that  an unshielded Peano continuum $K$ is a regular continuum.

We will use pseudo fibers and fibers recently introduced in \cite{NLC-1}, from which a numerical scale is developed that measures the extent to which such a continuum is locally connected.

More precisely, for any point $x\in K$, the pseudo fiber $E_x$  at $x$ consists of the points $y\in K$ such that there does not exist a simple closed curve $\gamma$ with $\gamma\cap \partial X$ a finite set, called a good cut, such that $x$ and $y$ lie in different component of $\mathbb{C}\setminus\gamma$; the fiber $F_x$ at $x$ is the component of $E_x$ that contains $x$.
By \cite[Proposition 4.2]{NLC-1}, for the locally connected unshielded continuum $K\subset\mathbb{C}$ every pseudo fiber $E_x$ equals the single point set $\{x\}$. Therefore, given any $x\in K$ and any open set $U\ni x$, we can choose for each $y\in K\setminus U$ a good cut $\gamma_y$ such that  $x$ and $y$ lie in different component of $\mathbb{C}\setminus\gamma_y$. Let $U_y, V_y$ be the component of $\mathbb{C}\setminus\gamma_y$ with $x\in U_y$ and $y\in V_y$. Then $\{V_y: y\in K\setminus U\}$ is an open cover of the compact set $K\setminus U$. Fix a finite sub-cover $\left\{V_{y_1},\ldots, V_{y_n}\right\}$. Then
\[U_x:=\bigcap_{i=1}^n U_{y_i}\]
is open in $K$, contains $x$, and is contained in $U$. Recall that, for $1\le i\le n$, the intersection $B_i:=\overline{U_{y_i}}\cap \overline{V_{y_i}}\cap K$ is contained in $\gamma_{y_i}\cap K$ hence is also a finite set. Since the boundary of $U_x$ in $K$ is defined to be the intersection
$\overline{U_x\cap K}\cap(K\setminus U_x)$ and is a subset of
\[\bigcup_{i=1}^n \left(\overline{U_x\cap K}\cap\overline{V_{y_i}}\cap K\right),\]
which is in turn a subset of $\bigcup_iB_i$ and hence is also a finite set. This verifies that $K$ is regular at $x$. Consequently, from flexibility of $x\in K$  we can infer that $K$ is a regular continuum.
\end{proof}

Note that another proof of this theorem can be found in~\cite[Lemma 2.7]{BO04}.

We mention that if the compactum $K\subset \mathbb{C}$ is not assumed to be unshielded, then the core decomposition of $K$ with respect to the finitely Suslinian property may not exist. For instance, the unit square $[0,1]^2$ is not finitely Suslinian and may be decomposed into vertical lines or into horizontal lines, both with a quotient space that is finitely Suslinian; while the only decomposition finer than these two decompositions is the one that decomposes $[0,1]^2$ into singletons. For cases of one-dimensional continua, one may  consider the locally connected continuum $K\subset\mathbb{C}$ given in~\cite[Example 14]{BCO13}. It admits two monotone decompositions $\mathcal{D}_1,\mathcal{D}_2$ such that  the quotients are finitely Suslinian. However, the only partition finer than both $\mathcal{D}_1$ and $\mathcal{D}_2$ is the trivial decomposition $\{\{x\}: x\in K\}$. Therefore, the core decomposition of $K$ with respect to the finitely Suslinian property does not exist, while the trivial decomposition $\{\{x\}: x\in K\}$ is the core decomposition of $K$ with respect to the property of being a Peano compactum.

We end up this section with an example of a continuum $K\subset\mathbb{C}$ having two properties. Firstly, the core decomposition $\mathcal{D}_K$ has an element $d$ such that at least two components of $\mathbb{C}\setminus d$ intersect $K$; secondly, the resulted quotient space $\mathcal{D}_K$ can not be embedded into the plane.

\begin{exam}\label{no-FS-core}
Let the compactum $K\subset\mathbb{C}$ be the union of the closure of the unit disk $\mathbb{D}=\{z\in\mathbb{C}: |z|<1\}$ and the spiral curve $L=\left\{\left(1+e^{-t}\right)e^{2\pi{\bf i}t}: t\ge0\right\}$. By routine works one may check that the core decomposition of $K$ with respect to the property of being a Peano compactum is exactly given by $\mathcal{D}_K=\left\{\{x\}: x\in \left(\mathbb{D}\cup L\right)\right\}\cup\{\partial \mathbb{D}\}.$
Clearly, the quotient space is the one-point union of a sphere with a segment, thus can not be embedded into the plane. \qed
\end{exam}


\section{Some Useful Lemmas}\label{lemmas}
The aim of this section is to prove Lemma~\ref{separation}. This lemma will be a great tool in order to obtain a characterization of the Sch\"onflies condition. Lemmas~\ref{brickwall} and~\ref{separating-gamma} and Theorems \ref{cut-wire} and \ref{torhorst} are known or may be deduced from known results. They are used to prove Lemma~\ref{separation}.

Let us start by recalling some notions and facts. For $X\subset \mathbb{C}$, we say that $X=A\cup B$ ($A,B\ne\emptyset$) is a \emph{separation of} $X$  if $\overline{A}\cap B=A\cap\overline{B}=\emptyset$. Moreover, given $M\subset X$ and $a,b\in X$, we say that $M $\emph{separates} $X$ \emph{between} $a$ \emph{and} $b$ if there is a separation $X\setminus M=A\cup B$ such that $a\in A$ and $b\in B$. Furthermore, remember that, if $x_0$ is a point in $X$ then the \emph{component} of $X$ containing $x_0$ is the maximal connected set $P\subset X$ with $x_0\in P$. The \emph{quasi-component} of $X$ containing $x_0$ is defined to be the set
\[Q=\{y\in X: \ \text{no\ separation}\ X=A\cup B\ \text{exists\ such\ that}\ x\in A, y\in B\}.\]
Equivalently, the quasi-component containing a point $p\in X$ may be defined as the intersection of all closed-open subsets of $X$ containing $p$. Any component is contained in a quasi-component, and the quasi-components coincide with the components whenever $X$ is assumed to be a compact set \cite[Section 47, Chapter II, Theorem 2]{Kuratowski68}.

The first of our lemmas arises from a basic property of the unit square $[0,1]^2$.

\begin{lemm}\label{brickwall}
Suppose that  $A\subset [0,1)\times[0,1]$ and $B\subset(0,1]\times[0,1]$ are disjoint closed sets. Then there exists a path in $[0,1]^2\setminus(A\cup B)$ starting from a point in $(0,1)\times\{0\}$ and leading to a point in $(0,1)\times\{1\}$.
\end{lemm}
The above lemma also appears as~\cite[Lemma 2.1]{LAT04}. In fact, it is a direct corollary of the following  property of $[0,1]^2$. Fix $p_i\in[0,1]\times\{i\}$ for $i=0,1$ and two disjoint closed subsets  $B_0,B_1$ of $[0,1]^2\setminus\{p_0,p_1\}$. Suppose that none of $B_0$ and $B_1$ separates $[0,1]^2$ between $p_0$ and $p_1$. Then $B_0\cup B_1$ does not separate $[0,1]^2$ between $p_0$ and $p_1$ \cite[p.347, $\S57$, II, Theorem 3]{Kuratowski68}. In such a case, we can find an arc in $[0,1]^2$ with end points $p_0,p_1$, which is disjoint from $B_0\cup B_1$.

We further recall two well known results, the {\bf Cut Wire Theorem} \cite[p.72, Theorem 5.2]{Nadler92} and {\bf Torhorst Theorem} \cite[p.512, \S 61, II, Theorem 4]{Kuratowski68}.

\begin{theo}[{\bf Cut Wire Theorem}]\label{cut-wire}
Let  $A$ and $B$ be closed subsets of a compactum $X$. If no connected subset of $X$ intersects both $A$ and $B$, then $X=X_1\cup X_2$ where $X_1$ and $X_2$ are disjoint closed subsets of $X$ with $A\subset X_1$ and $B\subset X_2$.
\end{theo}

\begin{theo}[{\bf Torhorst Theorem}]\label{torhorst}
If  $M\subset\bbC$ is a locally connected continuum then every component $R$ of $\bbC\setminus M$ has the following properties:
\\
{\rm(i)} the boundary $\partial R$ is a regular continuum containing no $\theta$-curve,
\\
{\rm(ii)} if $M$ contains no cut point, then $R$ is a disk and therefore $\partial R$ is a simple closed curve,
\\
{\rm(iii)} the closure $\overline{R}$ is a locally connected continuum.
\end{theo}

Then we obtain a lemma, that will be used in Remark \ref{separation-annulus} and Theorem \ref{quotient-con}.

\begin{lemm}[]\label{separating-gamma}
Given two components $P\ne Q$ of a compactum $X\subset\bbC$. If $Q$ lies in the unbounded component of \ $\bbC\setminus P$  there exists a simple closed curve $J\subset(\bbC\setminus X)$ such that $P$ is contained in the interior of $J$  and $Q$ in the exterior.
\end{lemm}
\begin{proof}
Fix a separation $X=E\cup F$ with $P\subset E,\ Q\subset F$ and a positive number $r$ less than the distance from $E$ to $F$. Then, cover $\bbC$ with small squares $T_{m,n}=T+mr+nr\left(\frac{1}{2}+{\bf i}\right) (m,n\in\bbZ)$, where  $T=\{t+s{\bf i}: 0\le t,s\le r\}$. Those squares form a tiling of $\bbC$, and no one of them intersects $E$ and $F$ both; moreover, the common part of two squares is either empty, or a vertical segment of length $r$, or a horizontal one of length $\frac{r}{2}$. See Figure \ref{bricks}.
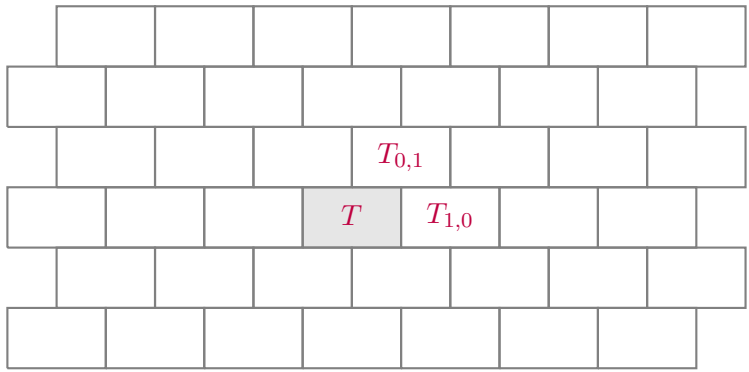
\begin{figure}[ht]
\begin{center}
\begin{tikzpicture}[x=1.618cm,y=1.0cm,scale=0.8]
\fill[gray!20] (3,2) -- (4,2) -- (4,3) -- (3,3) -- (3,2);

\foreach \i in {0,...,6}
\foreach \j in {0,2,4}
{
\draw[gray,thick] (\i,\j) -- (\i+1,\j) -- (\i+1,\j+1) -- (\i,\j+1) -- (\i,\j);
\draw[gray,thick] (\i+0.5,\j+1) -- (\i+1+0.5,\j+1) -- (\i+1+0.5,\j+1+1) -- (\i+0.5,\j+1+1) -- (\i+0.5,\j+1);
}

\draw[purple] (3.5,2.2) node[above]{$T$};  \draw[purple] (4,3.1) node[above]{$T_{0,1}$};
\draw[purple] (4.5,2.1) node[above]{$T_{1,0}$};
\end{tikzpicture}
\end{center}\vskip -0.5cm
\caption{A local patch of the tiling $\{T_{m,n}: m,n\in\bbZ\}$.}\label{bricks}
 \end{figure}
Therefore, if we consider the union of all the squares intersecting $E$ and denote by $A$ the component of this union that contains $P$, we easily see that the unbounded component $W$ of \ $\bbC\setminus A$ contains $Q$ and that the boundary $\partial W$ consists of finitely many segments and has no cut point. By Torhorst Theorem, we can infer that $\partial W$ is actually a simple closed curve. As $\partial W$ is disjoint from $X$, setting $J=\partial W$ completes our proof.
\end{proof}

\begin{rema} There are two things that are noteworthy. Firstly, since components of $X$ coincide with its quasi-components, the separation $X=E\cup F$ in the above proof may be chosen so that $E\supset P$ is contained in an arbitrarily small neighborhood of $P$. If we choose a small enough $r>0$ the resulted curve $J$ lies in a neighborhood of $P$ that is as small as we wish. Secondly, we mention that it is possible to avoid the general version of Torhorst Theorem in the above proof to infer  that $\partial W$ is a simple closed curve. Indeed, the result may be obtained by induction on the number of segments contained in $\partial W$. Taking distance from Torhorst Theorem is not meaningless, since the theory of core decomposition developed in the current paper aims at broader applications in the study of plane topology and dynamical systems. An example of such an application is to extend the Torhorst Theorem to the case of disconnected compacta in the plane.
\end{rema}



Now we are well prepared to state and prove a lemma which will be  very useful in the characterization of the
Sch\"onflies condition and in later discussions on the newly introduced notion of Peano compactum.

\begin{lemm}\label{separation}
Let $K\subset \bbC$ be a compactum and $U=U(L_1,L_2)$ the region bounded by two parallel lines $L_{1}$ and $L_{2}$. Let $U_i$ be the component of $\bbC\setminus(L_1\cup L_2)$ with $\partial U_i=L_i$. We have the following:
\item[(1)] If $K\cap L_1\neq\emptyset\neq K\cap L_2$ and no connected subset of $K$ intersects both $L_1$ and $L_2$, then there is a separation $K=A_1\cup A_2$ with $A_1$, $A_2$ compact sets such that $(\overline{U_i}\cap K)\subset A_i$.
\item[(2)] If $\overline{U}\setminus K$ has two components $R_1,R_2$ intersecting both $L_{1}$ and $L_{2}$, there exist two arcs $\alpha_i\subset R_i$ for $i=1,2$ such that $B\cap K$  has  a component $P$ intersecting both $L_{1}$ and $L_{2}$, where $B$ is the bounded component of $\overline{U}\setminus(\alpha_1\cup\alpha_2)$.
\item[(3)] If $\overline{U}\cap K$ has two components $Q_1, Q_2$ intersecting both $L_{1}$ and $L_{2}$, then $\overline{U}\setminus K$ has a component $R$, intersecting both $L_{1}$ and $L_{2}$, which contains an arc $\alpha$ separating $Q_1$ and $Q_2$ in $\overline{U}$.
\end{lemm}

\begin{rema}\label{separation-annulus}
The results of Lemma \ref{separation} still hold if we change $\bbC$ into $\hat{\bbC}$ and replace $L_1, L_2$ by two disjoint simple closed curves $J_1, J_2\subset\hat{\bbC}$. Then  the region $U(J_1,J_2)$ bounded by $J_1\cup J_2$ is an annulus. For Item (1), the argument is exactly the same. For Item (2), we just find an open arc $\alpha$ in $U(J_1,J_2)$ that connects a point on $J_1$ to a point on $J_2$. Clearly, the difference $\overline{U(J_1,J_2)}\setminus \overline{\alpha}$ is topologically homeomorphic with the unbounded strip bounded by two parallel lines. For Item (3), an application of Lemma \ref{separating-gamma} to $X=\overline{U}\cap K$ gives rise to such an open arc $\alpha$. Therefore, in the rest of this paper, we may use the results of Lemma \ref{separation} in each of the following two cases: (1) when $U\subset\bbC$ is an unbounded strip between two parallel lines;  (2) when $U\subset\hat{\bbC}$ is a topological annulus bounded by two disjoint simple closed curves.

\end{rema}

\begin{proof}[{\bf Proof for Lemma \ref{separation}}]

Item (1) follows  immediately from Theorem  \ref{cut-wire} (\cwt), while   the proof of Item (2) reads as follows.


For $i=1,2$ choose an arc $\alpha_{i}\subset R_{i}$ with endpoints $a_{i}\in L_1, b_{i}\in L_2$, such that $\alpha_i\setminus\{a_i,b_i\}\subset U$. Let $\beta$ be the arc on $L_1$ with endpoints $a_1, a_2$ and $\gamma$ the arc on $L_2$ with endpoints $b_1, b_2$. Then $\Gamma=\alpha_1\cup\beta\cup\alpha_2\cup \gamma$ is a simple closed curve. Let $W$ be the bounded component of $\bbC\setminus\Gamma$. Then $\overline{W}$ is a topological disk, which may be represented as a rectangle. See the following Figure \ref{brickwall-2} for relative locations of $\alpha_1,\alpha_2,\beta$, and $\gamma$.
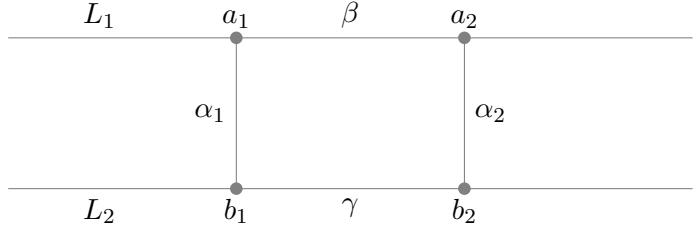
\begin{figure}[ht]
\begin{center}
\begin{tikzpicture}[scale=1]
\foreach \i in {0,1}
{
\draw[gray] (3*\i,0) -- (3*\i,2);
\draw[gray] (-3,2*\i) -- (6,2*\i);
\fill[gray] (3*\i,0) circle (0.5ex);
\fill[gray] (3*\i,2) circle (0.5ex);
}

\draw[] (-1.8,0) node[below]{$L_2$};
\draw[] (-1.8,2) node[above]{$L_1$};

\draw[] (0,0) node[below]{$b_1$};
\draw[] (0,2) node[above]{$a_1$};

\draw[] (3,0) node[below]{$b_2$};
\draw[] (3,2) node[above]{$a_2$};

\draw[] (1.5,0) node[below]{$\gamma$};
\draw[] (1.5,2) node[above]{$\beta$};

\draw[] (0,1) node[left]{$\alpha_1$};
\draw[] (3,1) node[right]{$\alpha_2$};

\end{tikzpicture}
\vskip -0.25cm
\caption{ The topological disk $\overline{W}$ is represented as a rectangle.}\label{brickwall-2}
\end{center}
\vskip -0.25cm
\end{figure}
Let $K^*=\overline{W}\cap K$. Then every component of $K^*$ is also a component of $\overline{U}\cap K$. And we only need to show that $K^*$ has a component intersecting both $L_1$ and $L_2$. This is guaranteed by Item (1).  Indeed, if on the contrary $K^*$ does not have such a component, we may use Item (1) to infer that $K^*$ has a separation $K^*=A\cup B$ with $A\cap\gamma=B\cap\beta=\emptyset$. Then, using Lemma \ref{brickwall} we may choose an arc $\alpha\subset\overline{W}$ that is disjoint from $K^*$ and connects a point on $\alpha_1$ to a point on $\alpha_2$. Therefore, $\alpha_1\cup\alpha\cup\alpha_2$ is a continuum lying in $\overline{U}\setminus K$. This is impossible, since the arcs $\alpha_1$ and $\alpha_2$ are  respectively contained in two components $R_1\ne R_2$ of $\overline{U}\setminus K$. This completes the proof of Item (2).

We now   prove Item (3).

To this end, we fix two lines $L', L''$, that are disjoint from $K$ and perpendicular to $L_1$ with  $L'\cap(L_1\cup L_2)=\{a_1', b_1'\}$ and  $L''\cap(L_1\cup L_2)=\{a_2', b_2'\}$, such that the rectangle $W$ with vertices $a_1', a_2',  b_2', b_1'$ contains $\overline{U}\cap K$. See Figure \ref{rectangle-W}.
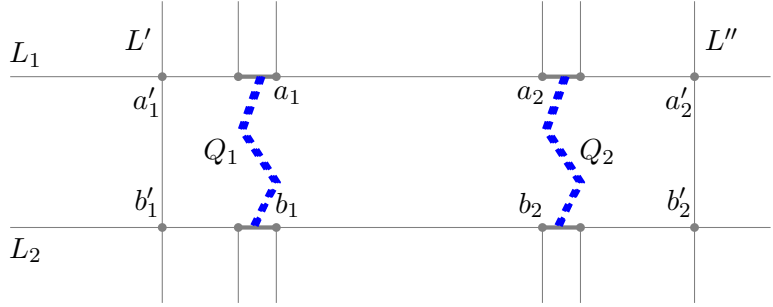
\begin{figure}[ht]
\begin{center}
\begin{tikzpicture}[scale=1.0]
\foreach \i in {0,1}
{
\draw[gray] (-1,2*\i)-- (9,2*\i);
\draw[gray] (1+7*\i,-1)-- (1+7*\i,3);
}

\foreach \i in {0,1}
{
\draw[gray] (2+4*\i,2) -- (2+4*\i,3);
\draw[gray] (2.5+4*\i,2) -- (2.5+4*\i,3);
\draw[gray] (2+4*\i,-1) -- (2+4*\i,0);
\draw[gray] (2.5+4*\i,-1) -- (2.5+4*\i,0);
\draw[gray,ultra thick] (2+4*\i,2) -- (2.5+4*\i,2);
\draw[gray,ultra thick] (2+4*\i,0) -- (2.5+4*\i,0);
\draw[blue,dashed,ultra thick] (2.27+4*\i,2) -- (2.27+4*\i-0.25,2-0.7) -- (2.27+4*\i+0.2, 2-1.4) -- (2.17+4*\i,0);
\draw[blue,dashed,ultra thick] (2.33+4*\i,2) -- (2.33+4*\i-0.25,2-0.7) -- (2.33+4*\i+0.2, 2-1.4) -- (2.2+4*\i,0);
\draw[blue,dashed,ultra thick] (2.30+4*\i,2) -- (2.30+4*\i-0.25,2-0.7) -- (2.30+4*\i+0.2, 2-1.4) -- (2.23+4*\i,0);
\fill[gray] (2+4*\i,2) circle (0.35ex);
\fill[gray] (2+4*\i,0) circle (0.35ex);
\fill[gray] (2.5+4*\i,2) circle (0.35ex);
\fill[gray] (2.5+4*\i,0) circle (0.35ex);
\fill[gray] (1+7*\i,2) circle (0.35ex);
\fill[gray] (1+7*\i,0) circle (0.35ex);
}

\draw[] (-0.8,0) node[below]{$L_2$};
\draw[] (-0.8,2) node[above]{$L_1$};
\draw[] (1,2.5) node[left]{$L'$};
\draw[] (8,2.5) node[right]{$L''$};

\draw[] (2.15,1.0) node[left]{$Q_1$};
\draw[] (6.35,1.0) node[right]{$Q_2$};

\draw[] (2.65,2) node[below]{$a_1$};
\draw[] (5.85,2) node[below]{$a_2$};
\draw[] (2.65,0) node[above]{$b_1$};
\draw[] (5.85,0) node[above]{$b_2$};

\draw[] (0.8,2) node[below]{$a_1'$};
\draw[] (7.8,2) node[below]{$a_2'$};
\draw[] (0.8,0) node[above]{$b_1'$};
\draw[] (7.8,0) node[above]{$b_2'$};
\end{tikzpicture}
\vskip -0.25cm
\caption{Relative locations of  $W$ and the components  $Q_1,Q_2$ of $\overline{U}\cap K$.}\label{rectangle-W}
\end{center}
\end{figure}
For $i=1,2$ let $I_i\subset L_1$ be the minimal arc that contains $L_1\cap Q_i$ and $J_i\subset L_2$ the minimal arc that contains $L_2\cap Q_i$. Then it is direct to check that $I_1\cap I_2=J_1\cap J_2=\emptyset$. We may assume with no loss of generality that $I_1$ separates $a_1'$ from $I_2$ in $L_1$, which then implies that  $J_1$ separates $b_1'$ from $J_2$ in $L_2$. See Figure \ref{rectangle-W} for relative locations of the arcs $I_1,I_2\subset L_1$ and $J_1, J_2\subset L_2$. Here we mark the right ends of $I_1, J_1$ respectively as $a_1$ and $b_1$, while the left ends of $I_2, J_2$ are marked as $a_2, b_2$ respectively.
Now, find a homeomorphism $h$ sending $\overline{W}$ onto $[0,1]^2$ such that $h(a_1)=(0,1), h(b_1)=(0,0), h(a_2)=(1,1)$, and $h(b_2)=(1,0)$. Then we may find a separation $h(K)=A\cup B$, with $h(Q_1)\subset A$ and $h(Q_2)\subset B$, such that $B\cap \{0\}\times[0,1]=A\cap \{1\}\times[0,1]=\emptyset$. Applying Lemma \ref{brickwall}, we may obtain an arc $\alpha_0\subset[0,1]^2$, that is disjoint from $A\cup B$ and connects a point on $[0,1]\times\{0\}$ to a point on $[0,1]\times\{1\}$. Then it is clear that $\alpha=h^{-1}(\alpha_0)$ is an arc in $\overline{W}\setminus K$ that connects a point on $L_1$ between $a_1,a_2$ to a point on $L_2$ between $b_1,b_2$. Therefore, the component of $\overline{U}\setminus K$ that contains $\alpha$ is what we want to find.
\end{proof}

The following lemma is a direct corollary of Lemma \ref{separation}.
\begin{lemm}\label{key-connection}
Let $K\subset \bbC$ be a compactum and $U=U(L_1,L_2)$ the region bounded by two parallel lines $L_{1}$ and $L_{2}$. We have the following:
\item[(1)] If $\overline{U}\setminus K$ has $m>2$ components intersecting both $L_{1}$ and $L_{2}$, then $\overline{U}\cap K$ has  at least $m-1$ components intersecting both $L_{1}$ and $L_{2}$.
\item[(2)] If $\overline{U}\cap K$ has $m>2$ components $Q_1,\ldots,Q_m$ intersecting both $L_{1}$ and $L_{2}$, then $\overline{U}\setminus K$ has at least $m-1$ components intersecting both $L_{1}$ and $L_{2}$, each of which separates at least two of the components $Q_1,\ldots,Q_m$.
\item[(3)] The {\em difference} $\overline{U}\setminus K$ has infinitely many components intersecting both $L_1$ and $L_2$ if and only if the intersection $\overline{U}\cap K$ has infinitely many components intersecting both $L_1$ and $L_2$.
\end{lemm}

\begin{rema}
Item (3) of the above lemma indicates that a compactum $K\subset\bbC$ satisfies the Sch\"onflies condition if and only if for any unbounded region $U=U(L_1,L_2)$ bounded by two parallel lines $L_1, L_2$ the intersection $\overline{U}\cap K$ has at most finitely many components intersecting both $L_1$ and $L_2$. This criterion is sometimes more helpful than the definition of Sch\"onflies condition, when the latter does not fit well into the setting. On the other hand, let $R_K$ be the Sch\"onflies relation on a compactum $K\subset\bbC$. Then Item (3) of Lemma \ref{key-connection} also implies that $x\ne y\in K$ are related under $R_K$ if and only if\, for any disjoint simple closed curves $J_1,J_2$ and the region $U=U(J_1,J_2)$ with $\partial U=J_1\cup J_2$, the difference $\overline{U}\setminus K$ has infinitely many components $\{Q_n\}$, intersecting $J_1$ and $J_2$ both, such that the limit $\lim\limits_{n\rightarrow\infty}Q_n$ under Hausdorff distance contains $\{x,y\}$. Such a criterion for $R_K$ includes a hint that the origin of $R_K$ is more or less related to the ancient result by Sch\"onflies~\cite[p.515, $\S61$, II, Theorem 10]{Kuratowski68}, on locally connected compacta in the plane.
\end{rema}

\section{Proofs for Theorems \ref{necessary} to \ref{sufficient}} \label{sec:necsuff}

Firstly, we copy the ideas of Sch\"onflies result \cite[p.515, $\S61$, II, Theorem 10]{Kuratowski68} and obtain a necessary condition for a planar compactum to be finitely Suslinian.

\begin{theo}\label{FS-necessary}
Given a finitely Suslinian compactum $K\subset\bbC$. If the sequence $R_1,R_2,\ldots$ of components of $\bbC\setminus K$ is infinite then the sequence of their diameters converges to zero.
\end{theo}
\begin{proof}
Suppose conversely that there exist $\epsilon>0$ and infinitely many integers $i_{1}<i_{2}<\cdots$ such that the diameter $\delta(R_{i_{n}})>3\epsilon$. For each component $R_{i_{n}}$, choose an arc $\alpha_{i_{n}}\subset R_{i_{n}}$ with diameter larger than $3\epsilon$. We may assume that $\alpha_{i_{n}}$ converges to $\alpha_{0}$ under Hausdorff distance. Here we have $\delta(\alpha_{0})\ge3\epsilon$. Choose two points $p',q'\in \alpha_{0}$ with $|p'-q'|=3\epsilon$. Then, we can fix two points $p_1,p_2$ in the interior of the segment $\overline{pq}$ with  $|p_1-p_2|=2\epsilon$.

Let $L_i$ be the line through $p_i$ which is perpendicular to $\overline{pq}$. Let $U$ be the unbounded strip satisfying $\partial U=L_1\cup L_2$. Since $\lim\limits_{n\rightarrow\infty}\alpha_{i_n}=\alpha_0$,
there exists an integer $N$ such that $\alpha_{i_{n}}$ intersects both $L_{1}$ and $L_{2}$ for all $n>N$. By Item (3) of Lemma \ref{key-connection}, there exist infinitely many components of $\overline{U}\cap K$ which intersect both $L_{1}$ and $L_{2}$. This contradicts the  condition that $K$ is finitely Suslinian.
\end{proof}

Secondly, we prove Theorem \ref{necessary} as follows.
\begin{proof}[{\bf Proof for Theorem \ref{necessary}.}]
The part for locally connected compacta follows from Sch\"onflies' result, see \cite[p.515, $\S61$, II, Theorem 10]{Kuratowski68}. We only consider the part for finitely Suslinian compacta.

Suppose on the contrary that there exist two parallel lines $L_1,L_2$ such that the difference $\overline{U}\setminus K$ has infinitely many components $R_1,R_2,\ldots$  intersecting each of $L_1$ and $L_2$. Here $U$ is the only component of $\mathbb{C}\setminus(L_1\cup L_2)$ with $\partial U=L_1\cup L_2$. By Item (3) of Lemma \ref{key-connection}, $\overline{U}\cap K$ has infinitely many components  intersecting each of $L_1$ and $L_2$. This contradicts the assumption that $K$ is finitely Suslinian.
\end{proof}

Then we continue to prove Theorem \ref{LC-criterion}.
\begin{proof}[{\bf Proof for Theorem \ref{LC-criterion}.}]
We only need to show that each continuum $K\subset\mathbb{C}$ satisfying the Sch\"onflies condition is locally connected. Suppose on the contrary that $K$ is not locally connected at a point $x_0\in K$. By definition of local connectedness \cite[p.227, $\S49$, I, Definition]{Kuratowski68}, there would exist a {\bf closed square} $V$ centered at $x_0$ 
such that the component $P_0$ of $V\cap K$ containing $x_0$  is not a neighborhood of $x_0$ with respect to the induced topology on $V\cap K$. In other words, there exists a sequence $\left\{x_{n}\right\}_{n=1}^{\infty}$ in $(V\cap K)\setminus P_0$ with $\lim\limits_{n\rightarrow\infty}x_{n}=x_0$ such that the components of $V\cap K$ containing $x_n$, denoted  $P_{n}$, are pairwise disjoint.

Recall that the hyperspace of all closed nonempty subsets of $V$ is a compact metric space under Hausdorff distance. Coming to an appropriate subsequence, if necessary, we may assume that $P_{n}$ converges to $P_{\infty}$ in Hausdorff distance. From this, we see that $P_{\infty}$ is a sub-continuum of $P_0$ and that the diameter of $P_{n}$, denoted $\delta(P_n)$, converges to $\delta(P_{\infty})$.

By connectedness of $K$, each $P_{n}$ must intersect $\partial V$ hence $P_\infty$ intersects $\partial V$. Since $P_n\rightarrow P_\infty$ under Hausdorff distance, we can pick some point $y_0\in(\partial V\cap P_\infty)$ and points $y_{n}\in(\partial V\cap P_{n})$ for all $n\ge1$  such that $y_0:=\lim\limits_{n\rightarrow\infty}y_{n}$.

Since $\partial V$ consists of four segments and contains the infinite set of points $\{y_n\}$,we may fix a line $L_1$ crossing infinitely many $y_n$, which necessarily contains $y_0$. Then, fix a line $L_2$  parallel to $L_1$ which separates $x_0$ from $y_0$, so that $x_0$ and $y_0$ lie in different components of $\mathbb{C}\setminus L_2$. Let $U$ be the strip bounded by $L_{1}$ and $L_{2}$. Obviously, there exists an integer $N$ such that $P_{n}$ intersects both $L_{1}$ and $L_{2}$ for $n\geq N$. Without loss of generality, we may assume that every $P_{n}$ intersects both $L_{1}$ and $L_{2}$.

It follows from \cwt\ that for all $n\ge1$ the intersection  $\overline{U}\cap P_{n}=V\cap \overline{U}\cap P_n$ has a component $Q_{n}$ intersecting both $L_{1}$ and $L_{2}$.
Therefore, by Item (3) of Lemma \ref{key-connection}, our proof will be completed if only we can show that for all but two integers $n\ge1$ the continuum $Q_n$ is also a component of $\overline{U}\cap K$. To this end, for all $n\ge1$ we may fix two points $a_n\in(L_1\cap Q_n)$ and $b_n\in(L_2\cap Q_n)$. Then, fix a line $L$ perpendicular to $L_1$ and disjoint from $K$.
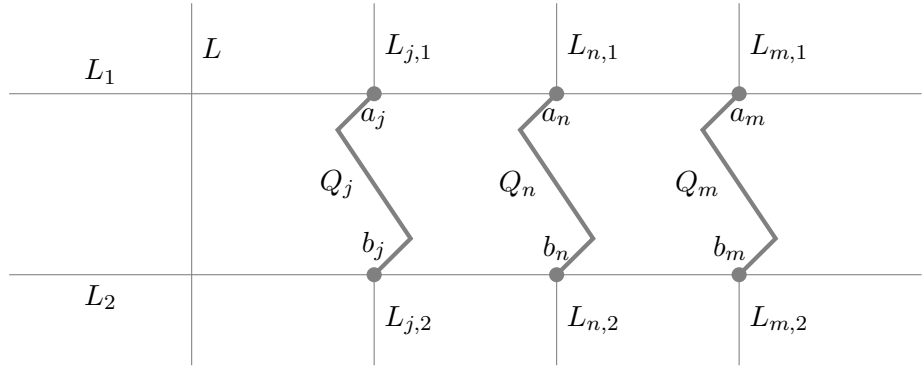
\begin{figure}[ht]
\begin{center}
\begin{tikzpicture}[scale=1.2]
\foreach \i in {0,1}
{
\draw[gray] (-2,2*\i)-- (8,2*\i);
}
\draw[gray] (0,-1)-- (0,3);

\foreach \i in {0,1,2}
{
\draw[gray] (2+2*\i,2) -- (2+2*\i,3);
\draw[gray] (2+2*\i,-1) -- (2+2*\i,0);
\draw[gray,ultra thick] (2+2*\i,2) -- (2+2*\i-0.4,2-0.4) -- (2+2*\i+0.4, 2-1.6) -- (2+2*\i,0);
\fill[gray] (2+2*\i,2) circle (0.5ex);
\fill[gray] (2+2*\i,0) circle (0.5ex);
}

\draw[] (-1,0) node[below]{$L_2$};
\draw[] (-1,2) node[above]{$L_1$};
\draw[] (0,2.5) node[right]{$L$};

\draw[] (2,2.5) node[right]{$L_{j,1}$};
\draw[] (4,2.5) node[right]{$L_{n,1}$};
\draw[] (6,2.5) node[right]{$L_{m,1}$};

\draw[] (2,-0.5) node[right]{$L_{j,2}$};
\draw[] (4,-0.5) node[right]{$L_{n,2}$};
\draw[] (6,-0.5) node[right]{$L_{m,2}$};

\draw[] (1.9,1.0) node[left]{$Q_j$};
\draw[] (3.9,1.0) node[left]{$Q_n$};
\draw[] (5.9,1.0) node[left]{$Q_m$};

\draw[] (2,1.95) node[below]{$a_j$};
\draw[] (4,1.95) node[below]{$a_n$};
\draw[] (6.1,1.95) node[below]{$a_m$};

\draw[] (2,0.05) node[above]{$b_j$};
\draw[] (4,0.05) node[above]{$b_n$};
\draw[] (5.9,0.05) node[above]{$b_m$};
\end{tikzpicture}
\caption{Relative locations of  $a_n, b_n$ and  $Q_j,Q_n,Q_m$.}\label{rays}
\end{center}
\end{figure}
See Figure \ref{rays}.
Now, we call $Q_n$  {\em a nearest component} (respectively {\em a furthest component})  if
\[{\rm dist}\left(a_n, L\right)<\min\left\{ {\rm dist}\left(a_j, L\right): j\ne n\right\}\quad \left(\text{or}\ {\rm dist}\left(a_n, L\right)>\max\left\{ {\rm dist}\left(a_j, L\right): j\ne n\right\}\right).\]
Clearly, there exist at most one nearest component and at most one furthest component. We claim that all the other $Q_n$ is also a component of  $\overline{U}\cap K$.
Actually, if $U_i$ denotes the component of $\bbC\setminus\overline{U}$ with $\partial U_i=L_i$ we can choose for all $n\ge1$ two rays $L_{n,1}\subset \overline{U_1}$ and $L_{n,2}\subset \overline{U_2}$  parallel to $L$ such that  $a_n\in L_{n,1}$ and $b_n\in L_{n,2}$.

If $Q_n$ is neither nearest nor furthest, then there exist two components $Q_j, Q_n$ satisfying
${\rm dist}\left(a_j, L\right)< {\rm dist}\left(a_n, L\right) < {\rm dist}\left(a_m, L\right)$. The above Figure \ref{rays} gives a simplified depiction for relative locations of $L, L_i$ and $a_j, a_n, a_m$.
In this case, we can use Jordan curve theorem to infer that $Q_n\setminus(L_1\cup L_2)$ is contained in a bounded component $W$ of $\bbC\setminus M$, where $M=\overline{a_ja_m}\cup \overline{b_jb_m}\cup Q_j\cup Q_m$ is a continuum which lies entirely in $V$. Therefore, $(W\cup M)\cap K$ is a subset of $V\cap K$; and we are able to choose a separation $\overline{W}\cap K=A\cup B$ with $Q_n\subset A$ and $(Q_j\cup Q_m)\cap A=\emptyset$. From this we can infer that $Q_n$ is also a component of $A$, which is disjoint from $B_1:=(\overline{U}\setminus W)\cap K$. That is to say, the intersection $\overline{U}\cap K$ is divided into two disjoint compact subsets, $A$ and $B\cup B_1$. Combining this with the fact that  $Q_n$ is a component of $A$, we readily see that $Q_n$ is a component of $\overline{U}\cap K$.
\end{proof}

Finally, we prove Theorem \ref{sufficient}.
\begin{proof}[{\bf Proof of Theorem \ref{sufficient}}]
Suppose that (1) and (2) hold and assume that $K$ does not satisfy the Sch\"onflies condition. Then there exists a region $U$ bounded by two parallel lines $L_{1}$ and $L_{2}$ such that $\overline{U}\cap K$ has infinitely many components $\{N_{m}\}$ which intersect both $L_{1}$ and $L_{2}$. Due to (2), there exists infinitely many $\{N_{m_{i}}\}$ in one component of $K$, denoted by $P$, and these $N_{m_i}$ are also in different components of $\overline{U}\cap P$. That is to say, $P$ does not satisfy the Sch\"onflies condition. By Theorem \ref{LC-criterion}, we reach a contradiction to the local connectedness of $P$. This verifies the ``if'' part.

To prove the ``only if'' part we assume that $K$ satisfies the Sch\"onflies condition and verify Conditions (1) and (2) as follows.

Given any component $P$ of $K$ and the region $U$ bounded by any two parallel lines $L_1$ and $L_2$, it is routine to check that every component of $\overline{U}\cap P$ is also a component of $\overline{U}\cap K$. Thus $\overline{U}\cap P$ has finitely many components intersecting both $L_1$ and $L_2$. By Theorem \ref{LC-criterion}, $P$ is locally connected.

If (2) is not true, then there exists an infinite sequence of sub-continua $\{N_m\}$ lying in distinct components of $K$, denoted by $Q_m$, such that their diameters $\delta(N_m)$ are greater than a positive constant $C$. Going to an appropriate subsequence, we may assume with no loss of generality that $\left\{N_{m}\right\}$ converges to a continuum $N_\infty\subset K$ under Hausdorff distance. Clearly, the diameter $\delta(N_\infty)$ is at least $C$. So we can choose two points $x,y\in N_\infty$ with $|x-y|=C$ and two parallel lines $L_1, L_2$ perpendicular to the line crossing $x,y$ and intersecting the interior of the segment $\overline{xy}$. Now we can see that all but finitely many $N_m\subset Q_m$ must intersect $L_1$ and $L_2$ at the same time. This implies that for all but finitely many integers $m\ge1$, there exists a component $P_m$ of $\overline{U}\cap Q_m$ intersecting both $L_1$ and $L_2$. Here $U$ is the strip bounded by $L_1\cup L_2$. For those integers $m$, the continua $P_m$ are each a component of $\overline{U}\cap K$, which contradicts the assumption that $K$ satisfies the Sch\"onflies condition.
\end{proof}

\section{The Quotient $\mathcal{D}_K$ is a Peano compactum}\label{sec:ps}
Throughout this section, we fix a compact set $K$ in the plane and denote by $\sim$  the Sch\"onflies equivalence relation given in Definition \ref{model}; moreover, the decomposition $\mathcal{D}_K$ of $K$ is made up of the equivalence classes $[x]=\{y\in K: x\sim y\}$, for $x\in K$. Then $\mathcal{D}_K$ is an upper semi-continuous decomposition. In the following proposition we show that its elements are each a subcontinuum of $K$, so that $\Dc_K$ is a monotone decomposition and $\pi: K\rightarrow \mathcal{D}_K$ a monotone map.

\begin{prop}\label{connected-class}
Given a compactum $K\subset\mathbb{C}$ and two points $x,y\in K$. If $(x,y)\in R_K$ then the class $[x]$ of $\sim$ contains a continuum $P_\infty\subset K$ with $\{x,y\}\subset P_\infty$. Consequently, the partition $\mathcal{D}_K$ is a monotone decomposition.
\end{prop}
\begin{proof} We add $\infty$ to $\bbC$ and construct a proof by considering $K$ as a subset of $\hat{\bbC}$. Firstly,
the assumption $(x,y)\in R_K$ says that there exist two disjoint simple closed curves,  $J_1\ni x$ and $J_2\ni y$,  such that the region  $U=U(J_1,J_2)$ with $\partial U=J_1\cup J_2$ has the following property: the common part $\overline{U}\cap K$ has infinitely many components $P_1,P_2,\ldots$ that intersect both $J_1$ and $J_2$ and converge under Hausdorff distance to a continuum $P_\infty$ containing $\{x,y\}$. Here we see that $U$ is an open annulus. We will show that for any $z\in P_\infty\cap U$ and $z_1\in P_\infty\setminus U=P_\infty\cap(J_1\cup J_2)$ the closure $\overline{R_K}$ of $R_K$ as a subset of the product $K\times K$ contains $(z,z_1)\in K^2$. Applying transitivity of the equivalence $\sim$, which contains $\overline{R_K}$, we can infer that $[x]=[z]\supset P_\infty$ holds for any $z\in P_\infty\cap U$.

We only need to verify that for any small $r>0$ there are two points $z', z_1'\in P_\infty$ with $|z-z'|=|z_1-z_1'|=r$ such that $(z',z_1')\in R_K$. And, we may assume that $z_1\in J_1$. If $z_1\in J_2$ the same argument also works.

By Sch\"onflies Theorem, we may consider $U$ as the annulus $\{1<|z|<2\}$. By Remark \ref{separation-annulus}, we may use Item (3) of Lemma \ref{separation} to find two arcs $\alpha_0, \beta_0\subset \overline{U}$ that are disjoint from $K$, such that $\overline{U}\setminus(\alpha_0\cup\beta_0)$ has two components one of which contains $\overline{U}\cap K$. Let $W$ be the closure of this component, which is a topological disk.  For $i=1,2$ let $\gamma_i=J_i\cap\partial W$. Then $\partial W$ consists of the four arcs $\alpha_0, \beta_0, \gamma_1$, and $\gamma_2$. Consider $W$ as the unit square, in a way that $\alpha_0$ corresponds to $\{0\}\times[0,1]$ and $\gamma_1$ to $[0,1]\times\{1\}$. See Figure \ref{location} for relative locations of $\alpha_0,\beta_0,\gamma_1$ and $\gamma_2$ in $W$.
\begin{figure}[ht]
\begin{center}
\begin{tikzpicture}[scale=1.3]
\footnotesize
 \pgfmathsetmacro{\xone}{0}
 \pgfmathsetmacro{\xtwo}{16}
 \pgfmathsetmacro{\yone}{0}
 \pgfmathsetmacro{\ytwo}{3.5}

\draw[gray,very thick] (\xone,\yone) -- (\xone+3*\xtwo/7,\yone) --  (\xone+3*\xtwo/7,\yone+6*\ytwo/7) --
(\xone,\yone+6*\ytwo/7) -- (\xone,\yone);

\draw[blue,very thick] (1.5*\xtwo/7,2.0*\ytwo/7) circle (0.3*\xtwo/7);
\draw[fill=blue] (1.5*\xtwo/7,2.0*\ytwo/7) circle (0.15em);
\draw[black] (1.5*\xtwo/7,2.0*\ytwo/7) node[anchor=north] {$z$};

\draw[black] (\xone,3*\ytwo/7) node[anchor=east] {$\alpha_0$};
\draw[black] (3*\xtwo/7,3*\ytwo/7) node[anchor=west] {$\beta_0$};
\draw[black] (1.75*\xtwo/7,6*\ytwo/7) node[anchor=south] {$z_1\in\gamma_1\subset J_1$};
\draw[line width=1pt,color=blue] (1.2*\xtwo/7,6*\ytwo/7) arc(180:360:0.3*\xtwo/7);

\draw[black] (1.75*\xtwo/7,0*\ytwo/7) node[anchor=north] {$\gamma_2\subset J_2$};
\draw[fill=blue] (1.5*\xtwo/7,6*\ytwo/7) circle (0.15em);

\foreach \p in {.4,.9,1.2}   \draw[gray,very thick] (0.1*\xtwo/7+\p*\xtwo/7,0*\ytwo/7) -- (0.1*\xtwo/7+\p*\xtwo/7,6*\ytwo/7);
\draw[black] (1.1*\xtwo/7,3.9*\ytwo/7) node[anchor=west] {$P_n'$};
\end{tikzpicture}
\end{center}\vskip -0.5cm
\caption{Relative locations of $\alpha_0, \beta_0$ and $z_1\in \gamma_1$ in $W$.}\label{location}
\vskip -0.25cm
\end{figure}
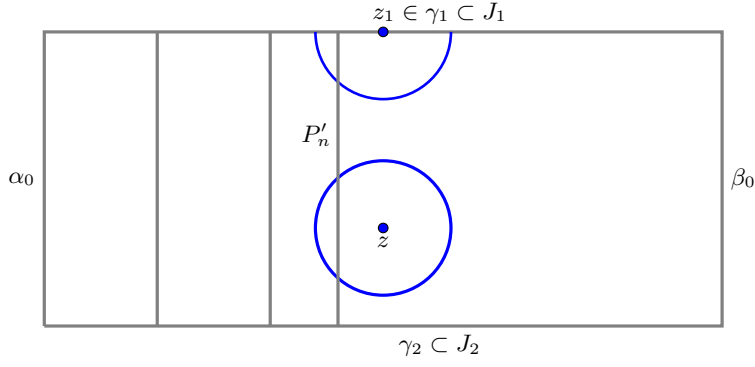

For any number $r>0$ smaller than $\frac12\min\left\{{\rm dist}(z,\partial W),\ {\rm dist}(z_1,\alpha_0\cup\beta_0)\right\}$,
let $D_r(z)$ be the open circular disk centered at $z$ with radius $r$ and denote its boundary as $\Gamma_r$. Let $D_r(z_1)$ be the open circular disk centered at $z_1$ with radius $r$ and denote by $\eta_r$ the semi-circle $\partial D_r(z_1)\cap \overline{W}$. Then $\Gamma=(\partial W\setminus D_r(z_1))\cup\eta_r$ is a simple closed curve disjoint from $\Gamma_r$. As $P_\infty\supset\{z,z_1\}$ and $\lim\limits_{n\rightarrow\infty}P_n=P_\infty$, we see that all but finitely many $P_n$ must intersect $\Gamma_r$ and $\eta_r$ both. Applying
Theorem \ref{cut-wire} (\cwt), we see that   $P_n\setminus\left(D_r(z)\cup D_r(z_1)\right)$ has a component $Q_n$ intersecting $\Gamma_r$ and $\eta_r$ at the same time. Here we recall that the region $U(\Gamma,\Gamma_r)$ bounded by $\Gamma\cup\Gamma_r$ is a subset of $W$.

Then, our proof  will be completed if only we can show that $Q_n$ is actually a component of $\overline{U(\Gamma,\Gamma_r)}\cap P_n$, which indicates that  $Q_n$ is also a component of $\overline{U(\Gamma,\Gamma_r)}\cap K$. To this end, we denote by $Q_n'$ the component of $\overline{U(\Gamma,\Gamma_r)}\cap P_n$ which contains $Q_n$. Then we have $Q_n'\subset P_n$ and $Q_n'\cap(D_r(z)\cup D_r(z_1))=\emptyset$. This indicates that $Q_n'$ is a sub-continuum of $P_n\setminus\left(D_r(z)\cup D_r(z_1)\right)$, which in turn indicates the equality $Q_n=Q_n'$.
\end{proof}


In the remaining part of this section, we assume that $\mathcal{D}_K$ is equipped with a metric $d$, which is compatible with the quotient topology. To prove Theorem \ref{quotient}, we start from a special case where the compactum $K$ is actually a continuum.

\begin{theo}\label{quotient-con}
If $K$ is a continuum then $\mathcal{D}_K$ is locally connected under quotient topology.
\end{theo}
\begin{proof} Assume on the contrary that the quotient space $\mathcal{D}_K$ is not locally connected at some point $\pi(x)$, where $x$ is a point in $K$ and $\pi(x)=[x]$. In the following we show that this would give rise to some thing that is absurd.

Under this assumption, there exists a closed ball $B_\varepsilon$ in $\mathcal{D}_K$ centered at $\pi(x)$ with radius $\varepsilon>0$ whose component containing the point $\pi(x)$ is not a neighborhood of $\pi(x)$. Let $Q$ be the component of $B_\varepsilon$ containing $\pi(x)$. Then we can find an infinite sequence $\pi(x_n)$ in $B_\varepsilon\setminus Q$ with $d(\pi(x_n),\pi(x))\rightarrow 0$ such that for $n\ne l$ the components $Q_n\ni \pi(x_n)$ and $Q_l\ni \pi(x_l)$ are disjoint.

Fix a point $\pi(y_n)$ on $Q_n\cap \partial B_\varepsilon$ for all $n\ge 1$. By coming to an appropriate subsequence, if necessary, we may assume that the two sequences $\{x_n\}, \{y_n\}$ in $K$ are convergent, with $x_n\rightarrow x_\infty$ and $y_n\rightarrow y_\infty$. Then, continuity of $\pi$ guarantees that $\pi(x_\infty)=\lim\limits_{n\rightarrow\infty}\pi(x_n)=\pi(x)$ and that
$\pi(y_\infty)=\lim\limits_{n\rightarrow\infty}\pi(y_n)$ belongs to the boundary of $B_\varepsilon$. Now, considered as subsets of $K$, we see that $\pi(x_\infty)=[x_\infty]$ and $\pi(y_\infty)=[y_\infty]$ are disjoint planar continua. In particular, we have $[x_\infty]\subset\pi^{-1}\left(B_\varepsilon^o\right)$ and
\[
[y_\infty]\subset E:=\pi^{-1}\left(\mathcal{D}_K\setminus B_\varepsilon^o\right)=K\setminus\pi^{-1}\left(B_\varepsilon^o\right),
\]
where $B_\varepsilon^o$ denotes the interior of $B_\varepsilon$. Since $\pi: K\rightarrow \mathcal{D}_K$ is a monotone map, the pre-images $\pi^{-1}(Q), \pi^{-1}(Q_n)$ are disjoint sub-continua of $K$. We may assume that $[y_\infty]$ lies in the unbounded component of $\bbC\setminus[x_\infty]$. (Otherwise, we exchange the places of $[x_\infty]$ and $[y_\infty]$.) Let $E_0$ be the part of $E$ lying in the unbounded component of $\bbC\setminus[x_\infty]$. Then $E_1=E\setminus E_0$ is compact and disjoint from $E_0$. Here it is possible that $E_1=\emptyset$; moreover, we may cover $E_0$ by a continuum $E_0^*$ that is disjoint from $[x_\infty]$.

By Lemma \ref{separating-gamma} there is a simple closed curve $J_1$, disjoint from $[x_\infty]\cup E_0^*$, such that $[x_\infty]$ is contained in the interior of $J_1$ and $E_0^*$ in the exterior. Using this lemma again, we may find a simple closed curve $J_2$, lying in the interior of $J_1$, such that $[x_\infty]$ is also contained in the interior of $J_2$. As $\Dc_K$ is upper semi continuous, we may choose $J_2$ in a small neighborhood of $[x_\infty]$ such that the following equation holds:
\begin{equation}\label{key-disjoint}
\pi(J_1\cap K)\cap\pi(J_2\cap K)=\emptyset.
\end{equation}
Let $D_1$ be the exterior of $J_1$, and $D_2$ the interior of $J_2$. Then $A=\bbC\setminus\left(D_1\cup D_2\right)$ is a closed topological annulus. Moreover, by the choices of $Q_n$, we see that $\pi^{-1}(Q_n)$ intersects $J_1, J_2$ both for all $n\ge1$. Since $\pi^{-1}(Q_n)$ is a component of $\pi^{-1}(B_\varepsilon)$ and since $E\cap A\subset\left[(E_1\cup E_0^*)\cap A\right]=\emptyset$, each component of $A\cap\pi^{-1}(Q_n)$ is also a component of $A\cap K$. Applying \cwt, we see that for all $n\ge1$ the intersection $A\cap\pi^{-1}(Q_n)$ has a component that intersects both $J_1$ and $J_2$. By the definition of Sch\"onflies relation $R_K$, we see that $R_K\cap(J_1\times J_2)\ne\emptyset$. This indicates that $\pi(J_1\cap K)\cap\pi(J_2\cap K)\ne\emptyset$, a contradiction to Equation (\ref{key-disjoint}).
\end{proof}

Then, we discuss the case where $K$ is a disconnected compactum.

\begin{proof}[{\bf Proof for Theorem \ref{quotient}}]
By Theorem \ref{quotient-con}, we just need to show that for any  $\varepsilon>0$ there are at most finitely many components of $\mathcal{D}_K$ with diameter greater than $\varepsilon$.

Otherwise, there is an infinite sequence $\left\{Q_j: j\right\}$ of components whose diameters  are greater than a constant $\varepsilon_0>0$. By monotonicity of the natural projection $\pi:K\rightarrow \mathcal{D}_K$, the preimages $P_j:=\pi^{-1}\left(Q_j\right)$ are components of $K$. By uniform continuity of $\pi$, the diameters of $P_j$ are greater than a constant $t_0>0$. By going to a subsequence, if necessary,we may assume that the sequence $\{P_j\}$ converges to a limit continuum $P_\infty$ under Hausdorff distance. Then, continuity of $\pi$ ensures that $Q_{j}\rightarrow \pi(P_\infty)$, which is a continuum of diameter  at least $\varepsilon_0$. So, we can fix two points $x_1,x_2\in P_\infty$ with $d(\pi(x_1),\pi(x_2))\ge\varepsilon_0$. Clearly, the two equivalence classes $[x_1], [x_2]$ are disjoint continua in the plane.
Since $\mathcal{D}_K=\{[z]: z\in K\}$ is an upper semi-continuous decomposition, we may fix two open sets $U_1\supset [x_1]$ and $U_2\supset [x_2]$ with disjoint closures satisfying the following equation
\begin{equation}\label{key-disjoint-0}
\pi(U_1)\cap\pi(U_2)=\emptyset.
\end{equation}
Let $J_i$ be the circle centered at $x_i$ with an arbitrary radius $r>0$ such that $J_i\subset U_i$. Let $W$ be the component of $\bbC\setminus(J_1\cup J_2)$ with $\partial W=J_1\cup J_2$. The containment $\{x_1,x_2\}\subset P_\infty$ implies that all but finitely many $P_j$ intersect both $J_1$ and $J_2$. Recall that every $P_j$ is a component of $K$. For each of those $P_j$ intersecting both $J_1$ and $J_2$, the intersection $\overline{W}\cap P_j$ has a component $Q_j$, which intersects $J_1$ and $J_2$ both. Those $Q_j$ are each a component of $\overline{W}\cap K$. Therefore, a subsequence of $\{Q_j\}$  converges to a limit continuum $Q_\infty$, which intersects both $J_1$ and $J_2$. For any $z_1\in Q_\infty\cap J_1$ and $z_2\in Q_\infty\cap J_2$, we have $\pi(z_1)=\pi(z_2)$. This contradicts Equation~(\ref{key-disjoint-0}).
\end{proof}

\section{$\mathcal{D}_K$ is the Core Decomposition of $K$ with Peano quotient}\label{sec:cd}

This section has a single aim, to prove Theorem \ref{finest}. And we {\bf only need to show} that $f(x_1)=f(x_2)$ for any  $x_1,x_2\in K$ with $(x_1,x_2)\in R_K$,
since $\sim$ is the smallest closed equivalence on $K$ containing $R_K$ and since $\{f^{-1}(y): y\in Y\}$ is also a monotone decomposition of $K$.

\begin{proof}[{\bf Proof for Theorem \ref{finest}}]
By Definition \ref{model}, we may fix two disjoint simple closed curves $J_i\ni x_i$ ($i=1,2$) such that $\overline{W}\cap K$ has infinitely many components $P_n$, each of which intersects both $J_1$ and $J_2$, such that the sequence $\{P_n\}$ converges to a continuum $P_\infty\supset\{x_1,x_2\}$ under Hausdorff distance. Here $W$ is the only component of $\bbC\setminus(J_1\cup J_2)$ with $\partial W=J_1\cup J_2$.

By applying an appropriate M\"obius transformation, if necessary, we may assume with no loss of generality that $J_1$ separates $J_2$ from infinity.
Since its topology is equivalent to a closed annulus, we may consider $\overline{W}$ to be $A=\{z: 1\le |z|\le 2\}$. By Item (3) of Lemma \ref{separation}, we may fix two arcs  $\alpha_0,\beta_0$ in $A$, which are disjoint from $K$, such that one of the two components of $A\setminus(\alpha_0\cup\beta_0)$ is disjoint from $K$. Denote the other one as $D_0$. Then $\overline{D_0}$ is a topological disk, whose boundary includes $\alpha_0, \beta_0$, one arc on $J_1$ and another on $J_2$.

Suppose on the contrary that $f(x_1)\ne f(x_2)$. Then there would exist a positive number $\varepsilon_0<\frac12 \rho(f(x_1),f(x_2))$ such that for $i=1,2$ the balls  $B_i\subset Y$,  centered at $f(x_i)$ with radius $\varepsilon_0$, have disjoint closures, {\em i.e.} $\overline{B_1}\cap \overline{B_2}=\emptyset$.
In the rest of our proof we set $U_i=f^{-1}(B_i)$.

The convergence $f(P_n)\rightarrow f(P_\infty)$ under Hausdorff distance ensures that  all but finitely many $f(P_n)$ are of diameter greater than $2\varepsilon_0$.
Since $Y$  is a Peano compactum, all but finitely many $f(P_n)$ must be contained  in a single component  of $Y$, denoted $Y_0$, which also contains $f(P_\infty)$. We may assume that every $f(P_n)$ is entirely contained in $Y_0$. Since $f: K\rightarrow Y$ is monotone, the pre-image $f^{-1}(Y_0)$ is a component of $K$, denoted as $L$. Clearly, the above $P_n$ are all contained in $L$. Therefore, we have $P_\infty\subset L$ and  $(x_1,x_2)\in R_L$.

To complete our proof, we will induce a contradiction to local connectedness of $Y_0$, by showing that $Y_0$ is not locally connected at some point on $f(P_\infty)$. More precisely, let $L_0=L\setminus(U_1\cup U_2)$; then we will find a point $y^\#\in f(P_\infty)\subset Y_0$ such that $f(L_0)$ is a neighborhood of $y^\#$ in $Y_0$, while the component of $f(L_0)$ containing $y^\#$ is not (a neighborhood of $y^\#$ in $Y_0$).

Repeatedly applying Item (3) of Lemma \ref{separation}, we may find an infinite sequence of arcs $\alpha_i\subset\overline{D_0}$, each of which separates two of the components $P_1,P_2,\ldots$ in $\overline{D_0}$. By choosing an appropriate subsequence and change the notations of $\alpha_0$ and $\beta_0$, if necessary, we may assume that for all $n\ge1$ the arc $\alpha_n$ separates $\alpha_{n-1}$ and $\beta_0$ in $\overline{D_0}$. For $n\ge1$ denote by $D_n$
the only bounded component of
$\bbC\setminus\left(\partial D_0\cup\alpha_n\cup\alpha_{n-1}\right)$
whose boundary contains $\alpha_{n-1}\cup\alpha_n$. The closure $\overline{D_n}$ is a topological disk that contains at least one of the components $P_1,P_2,\ldots$. Denote this component as $Q_n$. Then $\lim\limits_{n\rightarrow\infty}Q_n
=\lim\limits_{n\rightarrow\infty}\overline{D_n}=P_\infty$
under Hausdorff distance. See Figure \ref{D_0} for relative locations of $\alpha_0,\beta_0$ and $Q_n$ in $\overline{D_0}$, which is to be considered as $[0,1]^2$ such that $\alpha_0$ corresponds to $\{0\}\times[0,1]$ and $\beta_0$ corresponds to $\{1\}\times[0,1]$.
\begin{figure}[ht]
\begin{center}
\begin{tikzpicture}[scale=1]
\foreach \i in {0,1}
{
\draw[gray] (-1,4*\i)-- (10,4*\i);
\draw[gray] (11*\i-1,0)-- (11*\i-1,4);
\fill[gray] (5,4*\i) circle (0.5ex);
}

%
\draw[gray,ultra thick] (2,0) -- (2.25,1) -- (1.75, 3) -- (2,4);
\draw[] (2,2) node[left] {$Q_n$};
\draw[] (1.75,3) node[left] {$x_{n,1}$};
\draw[] (2.25,1) node[left] {$x_{n,2}$};

\draw[gray,ultra thick] (3,0) -- (3.25,1) -- (2.75, 3) -- (3,4);
\draw[gray,ultra thick] (3.5,0) -- (3.75,1) -- (3.25, 3) -- (3.5,4);
\draw[gray,ultra thick] (5,0) -- (5.15,1) -- (4.85, 3) -- (5,4);

\draw[] (-1,2) node[left] {$\alpha_0$};
\draw[] (10,2) node[right] {$\beta_0$};


\draw[] (3,2) node[left] {$\widetilde{\alpha_j}$};
\draw[] (3.5,2) node[right] {$\widetilde{\alpha_{j+1}}$};
\fill[gray] (2.25,1) circle (0.5ex);
\fill[gray] (1.75,3) circle (0.5ex);
\draw[gray, ultra thick] (2.25,1) -- (5,0);
\draw[gray, ultra thick] (1.75,3) -- (5,4);

\draw[] (5,0) node[below] {$x_2$};
\draw[] (5,4) node[above] {$x_1$};

\fill[gray] (2.822,3.33) circle (0.5ex);
\fill[gray] (3.375,3.5) circle (0.5ex);
\fill[gray] (3.16,0.669) circle (0.5ex);
\fill[gray] (3.625,0.5) circle (0.5ex);
\end{tikzpicture}\vskip -0.5cm
\caption{Relative locations of $\alpha_0, \beta_0$, and $Q_n$  in $\overline{D_0}$.}\label{D_0}
\end{center}\vskip -0.5cm
\end{figure}
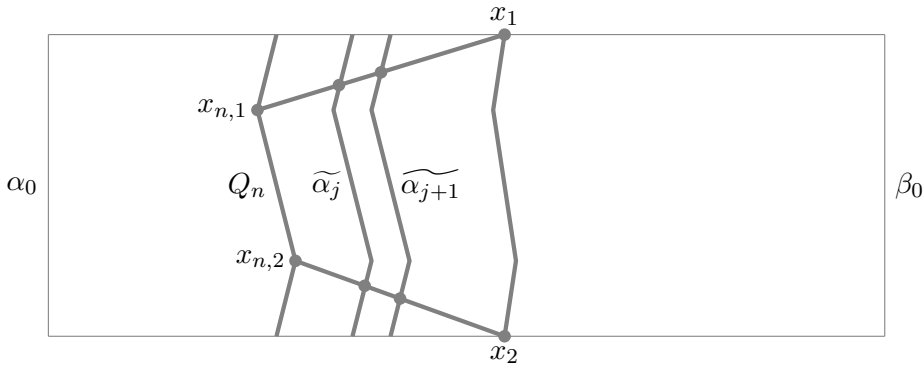

Now, fix two disks $D(x_i,r)$ centered at $x_i$ with radius $r$, for very small $r>0$ such that $D(x_i,r)\subset U_i$ for $i=1,2$. Since $Q_n\rightarrow P_\infty$ under Hausdorff distance, we may find two points $x_{n,i}$ in $(0,1)^2\cap Q_n$ for a large $n\ge1$ such that both $|x_{n,1}-x_1|$ and $|x_{n,2}-x_2|$ are smaller than $\min\left\{\frac{1}{2},r\right\}$.

Then, the segments $\gamma_{n,i}:=\overline{x_{n,i}x_i}$ for $i=1,2$ are disjoint; moreover, both $\gamma_{n,1}$ and $\gamma_{n,2}$ intersect all but finitely many of the topological disks $D_j$, that are constructed as above. In particular, each of $\gamma_{n,1}, \gamma_{n,2}$ intersects both $\alpha_j$ and $\alpha_{j+1}$ for all $j>n$. For each of those $j$, let $a_j$ be the last point of $\gamma_{n,1}$ that leaves $\alpha_j$ and $b_j$ the first point of $\gamma_{n,1}$ after $a_j$ that lies on $\alpha_{j+1}$;  let $c_j$ be the last point of $\gamma_{n,2}$ that leaves $\alpha_j$ and $d_j$ the first point of $\gamma_{n,2}$ after $c_j$ that lies on $\alpha_{j+1}$.
Then, the segments $\overline{a_jb_j}, \overline{c_jd_j}$, the arc $\widetilde{\alpha_j}\subset\alpha_j$ connecting $a_j$ to $b_j$ and the arc $\widetilde{\alpha_{j+1}}\subset\alpha_{j+1}$  connecting $c_j$ to $d_j$  form a simple closed curve, denoted $\Gamma_j$. Since the disks $D_j$ are disjoint, so are the disks $\Delta_j\subset D_j$ that are bounded by $\Gamma_j$.
For all the integers $j,j'>n$, the following  observations are immediate and will be useful in the rest of our proof.
\begin{itemize}
\item[(a)] The intersection $L\cap\left(\overline{a_jb_j}\cup \overline{c_jd_j}\right)=L\cap(\partial\Delta_j)=L\cap\Gamma_j$ is contained in $D(x_1,r)\cup D(x_2,r)$, which is a subset of $U_1\cup U_2=f^{-1}(B_1)\cup f^{-1}(B_2)$.

\item[(b)] \cwt\ implies that every $\overline{\Delta_j}\cap Q_j$ has a component $M_j$ that intersects the segments $\overline{a_jb_j}$ and $\overline{c_jd_j}$ at the same time. Since $L\cap\overline{a_jb_j}$ and $L\cap\overline{c_jd_j}$ are respectively contained in $U_1$ and $U_2$, the difference $M_j\setminus(U_1\cup U_2)$ has a component $N_j$ that intersects $\partial U_1$ and $\partial U_2$ both. Thus  $f(N_j)\cap \partial B_i\ne\emptyset$ for $i=1,2$.
\item[(c)] The compactum $L_0=L\setminus(U_1\cup U_2)=L\setminus\left(f^{-1}(B_1)\cup f^{-1}(B_2)\right)$ contains $N_j$. Moreover, no  component of  $f(L_0)=Y_0\setminus(B_1\cup B_2)$ contains $f(N_j)\cup f(N_{j'})$ for $j\ne j'$, since the pre-image of such a component $A_0$ would be a sub-continuum of $L_0$ that intersects $N_j$ and $N_{j'}$ both and hence intersects $\Gamma_j$. By the containment $(L\cap\Gamma_j)\subset (U_1\cup U_2)$ obtained in Observation (a), this contradicts the equality $L_0=L\setminus(U_1\cup U_2)$.
\end{itemize}
Now, the connectedness of $N_j$ allows us to choose a point $z_j\in N_j$ such that the distance from $f(z_j)$ to $\overline{B_1}$ equals that from $f(z_j)$ to $\overline{B_2}$.
Let $z_\infty$ be a limit point of $\{z_j\}$. Then $z_\infty\in P_\infty$ and the distance from $y^\#:=f(z_\infty)$ to $\overline{B_1}$ equals that from $y^\#$ to $\overline{B_2}$. Thus the compactum $f(L_0)$ is a neighborhood of $y^\#$ in $Y_0$. It follows from the above Observation (c) that no two of the points $f(z_j)\in N_j$ are contained in a single component of $f(L_0)$. Therefore, the component of $f(L_0)$ containing $y^\#$ is not a neighborhood of $y^\#$ in $Y_0$. This completes the proof.
\end{proof}

\section{How about compacta $K$ that may not be planar ?}

This section concerns the core decomposition with Peano quotient for a compactum that may not be planar. Here we have two aims. The first is to construct a concrete continuum $K$ in $\bbR^3$, such that the core decomposition with Peano quotient does not exist.

\begin{exam}\label{no-CD-ex}
Let $\mathcal{C}\subset[0,1]$ be the Cantor ternary set. Let $K$ be the union of three compact sets: $\mathcal{C}\times[0,1]^2$, $[0,1]\times\{0\}\times[0,1]$, and $[0,1]^2\times\{0\}$. See the following picture for a perspective depiction of $K$.
\begin{figure}[ht]
\vskip-0.25cm
\begin{center}
\begin{tikzpicture}[x=1cm,y=1cm,scale=1.25]

	\coordinate (P1) at (-9cm,2cm); 
	\coordinate (P2) at (5cm,2cm); 

	\coordinate (A1) at (0cm,0cm); 
	\coordinate (A2) at (0cm,-2cm); 

	\coordinate (A3) at ($(P1)!.8!(A2)$); 
	\coordinate (A4) at ($(P1)!.8!(A1)$);

	\coordinate (A6) at ($(P2)!.64!(A2)$);
	\coordinate (A5) at ($(P2)!.64!(A1)$);

	\coordinate (A8) at
	  (intersection cs: first line={(A5) -- (P1)},
			    second line={(A4) -- (P2)});
	\coordinate (A7) at
	  (intersection cs: first line={(A6) -- (P1)},
			    second line={(A3) -- (P2)});


	\fill[green!38.2] (A2) -- (A3) -- (A7) -- (A6) -- cycle; 
	\fill[green!38.2] (A3) -- (A4) -- (A8) -- (A7) -- cycle; 
    \draw[thick,dashed] (A3) -- (A7);


	\foreach \i in {0,1,2,3}
	{
	\coordinate (B2) at ($(P2)!0.64+.12*\i!(A2)$);
	\coordinate (B1) at ($(P2)!0.64+0.12*\i!(A1)$);

	\coordinate (B4) at
	  (intersection cs: first line={(B1) -- (P1)},
			    second line={(A4) -- (P2)});
	\coordinate (B3) at
	  (intersection cs: first line={(B2) -- (P1)},
			    second line={(A3) -- (P2)});
\fill[purple,opacity=0.25] (B1) -- (B2) -- (B3) -- (B4) -- cycle; 
\draw[black,thick] (B1) -- (B2) -- (B3) -- (B4) -- cycle;
	}



\draw[fill=black] (A3) circle (0.05em)  node[left] {(1,0,0)};
\draw[fill=black] (A6) circle (0.05em)  node[right] {(0,1,0)};
\draw[fill=black] (A8) circle (0.05em)  node[above] {(0,0,1)};
\draw[thick,black] (A1) -- (A5);
\draw[thick,black] (A2) -- (A6);
\draw[thick,black] (A4) -- (A8);
\end{tikzpicture}
\end{center}
\vskip -0.75cm
\caption{A simplified approximation of the continuum $K$.}
\vskip -0.25cm
\end{figure}
It is routine to check that $K$ is a continuum and is semi-locally connected everywhere. Consider the projections $\pi_1(x_1,x_2,x_3)=(x_1,x_2)$ and $\pi_2(x_1,x_2,x_3)=(x_1,x_3)$ from $K$ onto $[0,1]^2$.
One may verify that the two collections of pre-images $\Dc_i=\{\pi_i^{-1}(x): x\in[0,1]^2\}$ are each a monotone decomposition of $K$, whose elements are either singletons or segments of length $1$. Moreover, the only decomposition of $K$ finer than both $\Dc_1$ and $\Dc_2$ is the decomposition into singletons. Therefore, $K$ does not allow a core decomposition with Peano quotient, since $K$ itself is continuum that is not locally connected.
\end{exam}

The second aim is to show that two topologically equivalent compacta $K,L$ either have no Peano model or have topologically equivalent ones. Namely, we have the following basic result.
\begin{theo}
\label{topological-CD}
Let $K$ be a compactum such  that the core decomposition $\Dc_K^{PC}$  exists. If $h: K\rightarrow L$ is a homeomorphism of $K$ onto $L$ then $\Dc_L^{PC}$ exists and equals $\left\{f(d): d\in\Dc_K^{PC}\right\}$.
\end{theo}
\begin{proof}
Let $\pi: K\rightarrow\Dc_K^{PC}$ be the natural projection. Let $\rho$ be a metric that is compatible with the quotient topology on $\Dc_K^{PC}$. Then $(\Dc_K^{PC},\rho)$ is a Peano compactum.
Since $h: K\rightarrow L$ is a homeomorphism, the decomposition $\Dc_L=\left\{h(d): d\in\Dc_K^{PC}\right\}$ is monotone, and the composition $\pi\circ h^{-1}: L\rightarrow\Dc_K^{PC}$ is a monotone map of $L$ onto a Peano compactum. If $\Dc'$ is an arbitrary monotone decomposition of $L$ with Peano quotient the quotient space, also denoted as $\Dc'$, is a Peano compactum under a compatible metric. Let $\pi': L\rightarrow \Dc'$ be the natural projection. Then the composition $\pi'\circ h: K\rightarrow\Dc'$ is a monotone map of $K$ onto a Peano compactum $\Dc'$. Since $\Dc_K^{PC}$ is the core decomposition of $K$ with Peano quotient, it necessarily refines the decomposition $\left\{h^{-1}(d'): d'\in\Dc'\right\}$. This indicates that $\Dc_L$ refines $\Dc'$. By the flexibility of $\Dc'$, we readily see that $\Dc_L$ is indeed the core decomposition of $L$ with Peano quotient. This ends our proof.
\end{proof}

\begin{rema}
After Theorem \ref{topological-CD}, one may wonder {\em under what conditions} the core decomposition with Peano quotient exists for a {\em non-planar} compactum $K$. Another question of some interest is whether the core decomposition with Peano quotient exists for any compactum $K$ lying on the torus, or on a general closed surface.
\end{rema}

\bibliographystyle{plain}

\end{document}